\documentclass[11pt]{article}
\usepackage{cite}
\usepackage{amsmath,amssymb,amsfonts}
\usepackage{amsthm}
\usepackage{algorithmic}
\usepackage{graphicx}
\usepackage{textcomp}
\usepackage{bbm}
\def\BibTeX{{\rm B\kern-.05em{\sc i\kern-.025em b}\kern-.08em
    T\kern-.1667em\lower.7ex\hbox{E}\kern-.125emX}}
    
 \usepackage{authblk}
    
\usepackage{tikz}
\usetikzlibrary{calc,patterns,angles,quotes}
\usepackage{cases}
\usepackage{float}
\usepackage{subcaption}
\usepackage{xcolor}

\definecolor{color_selva}{RGB}{34,139,34}
\definecolor{lightgray}{gray}{0.9}

\newtheorem{definition}{Definition}
\newtheorem{lemma}{Lemma}
\newtheorem{theorem}{Theorem}

\newtheorem{proposition}{Proposition}

\newtheorem{remark}{Remark}

\newtheorem{question}{Question}

\providecommand{\keywords}[1]
{
  \small	
  \textbf{\textit{Keywords---}} #1
}

\newcommand{\Z}{\mathbb{Z}}

\newcommand{\C}{\mathbb{C}}
\newcommand{\R}{\mathbb{R}}
\newcommand{\D}{\mathcal{D}}

\begin{document}
\title{Weighted Inequalities for $t$-Haar multipliers}
\author[1]{Daewon Chung}
\author[2]{Weiyan Huang}
\author[3]{Jean Carlo Moraes}
\author[4]{Mar\'ia Cristina Pereyra}
\author[5]{Brett D. Wick}
\affil[1]{Keimyung University, Daegu, S. Korea}
\affil[2,5]{Washington University, Saint Louis, MO}
\affil[3]{Universidade Federal do Rio Grande do Sul, Porto Alegre, Brazil}
\affil[4]{University of New Mexico, Albuquerque, NM}

\maketitle

\begin{abstract}
In this paper, we provide necessary and sufficient conditions on a triple of weights $(u,v,w)$ so that the $t$-Haar multipliers $T^t_{w,\sigma}$,  $t\in \R$,
are uniformly (on the choice of signs $\sigma$) bounded from $L^2(u)$ into $L^2(v)$. These dyadic operators have symbols  $s(x,I)=\sigma_I\,(w(x)/\langle w\rangle_I)^t$  which are functions of the  space variable $x\in\R$ and the frequency variable $I\in \mathcal{D}$,    making them   dyadic analogues of pseudo-differential operators.  Here $\mathcal{D}$ denotes the dyadic intervals,  $\sigma_I=\pm1$,  and  $\langle w\rangle_I$ denotes the integral average of $w$ on $I$. When $w\equiv 1$ we have the martingale transform and our conditions recover the known  two-weight necessary and sufficient conditions of Nazarov, Treil and Volberg.
We also show how these conditions are simplified when $u=v$. In particular,  the martingale  one-weight and the $t$-Haar multiplier unsigned and unweighted  (corresponding to $\sigma_I\equiv 1$ and $u=v\equiv 1$) known results are recovered or improved. We also obtain necessary and sufficient testing conditions of Sawyer type for the two-weight boundedness of a single variable Haar multiplier similar to those known for the martingale transform.
\end{abstract}


\section{Introduction}
\label{sec:introduction}

We are interested in seeking   necessary and sufficient conditions on a triple of weights $(u,v,w)$ so that the $t$-Haar Multiplier $T^t_{w,\sigma}$, $t\in\R$, defined in the next paragraph, maps $L^p(u)$ into $L^p(v)$ uniformly on the choice of signs $\sigma$.  Given a weight $u$, namely an almost everywhere positive locally integrable function on $\R$, $L^p(u)$ is the space of all functions such that $ \int_{\mathbb{R}} |f(x)|^p u(x)\, dx$ is finite. 

For a weight $w$ and $t\in\R$, the \emph{$t$-Haar multiplier} $T^t_{w,\sigma}$, acting on locally integrable functions $f$,  is defined formally as
\[ T^t_{w,\sigma} f(x) = \sum_{I\in \mathcal{D}} \sigma_I\left (\frac{w(x)}{\langle w\rangle_I}\right )^t  \langle f,h_I\rangle \,h_I(x),\]
where $\mathcal{D}$ denotes the dyadic intervals on $\R$, $\sigma_I=\pm 1$, $h_I$ is the $L^2$-normalized Haar function associated to the interval $I$, $\langle w\rangle_I$ 
denotes the integral average of the function $w$ on the interval $I$, and $\langle f,g\rangle$
denotes the inner product on $L^2(\R)$.  The Haar functions indexed on the dyadic intervals form an orthonormal basis on $L^2(\R)$, see~\cite{Ha},  and an unconditional basis on $L^p(\R)$ for $1<p<\infty$, see \cite{Pa, Mar}. 

The \emph{variable Haar multipliers} with symbol $s(x,I)$, acting on locally integrable functions $f$, are defined  formally by
\[ T_sf(x)=\sum_{I\in\mathcal{D}} s(x,I)\langle f,h_I\rangle h_I(x).\]
 They are a Haar analogue of pseudo-differential  operators with symbol $\sigma(x,\xi)$, a function of the spatial variable $x$ and the frequency variable $\xi$, given by
 \[ \Psi_{\sigma}f(x)= \int_{\R} \sigma(x,\xi ) \widehat{f}(\xi) e^{2\pi ix\xi}\,  d\xi,\]
 where the Fourier transform of a  function $f$ in the Schwartz class is defined to be $\widehat{f}(\xi) := \int_{\R}f(x)e^{-2\pi ix\xi}dx= \langle f, e_{\xi}\rangle$ and $e_{\xi}(x)=e^{2\pi ix\xi}$.
  The Fourier modes have been  replaced by the Haar basis, and the symbols $s(x,I)$ now depend on the spatial variable $x$ and the dyadic frequency variable $I$.   The $t$-Haar multipliers are variable Haar multipliers with symbols  $s(x,I)=\sigma_I (w(x)/\langle w\rangle_I )^t$ where $w$ is a weight, $\sigma_I=\pm1$, and $t\in\R$.

When $w\equiv 1$ then  $T^t_{1,\sigma}$ has symbol $s(x,I)=\sigma_I$ independent of $t$ (and $x$) and is denoted simply $T_{\sigma}$. The operator $T_{\sigma}$ is the martingale transform known to be bounded on $L^p(\R )$ for $1<p<\infty$, see \cite{Bur}, and bounded on $L^p(w)$ if and only if $w$ is in the \emph{dyadic Muckenhoupt} $A^d_p$ class for $1<p<\infty$, see \cite{TV}, namely
\begin{equation}
[w]_{A_p^d}:=\sup_{I\in\mathcal{D}} \bigg(\frac{1}{|I|}\int_Iw(x)\,dx\bigg) \bigg(\frac{1}{|I|}\int_Iw^{\frac{-1}{p-1}}(x)\,dx\bigg)^{p-1}<\infty. \label{def:Ap}
\end{equation}
Optimal quantitative estimates  are known for the martingale transform, see \cite{Wi} for the linear bound on $L^2(w)$, \cite{PetV} for $p>2$  and  \cite{DGPPet} for $1<p<2$. More precisely, there is a positive constant $C_p$ independent of the choice of signs $\sigma$  such that for all $f\in L^p(w)$,
\[ \|T_{\sigma} f\|_{L^p(w)} \leq C_p [w]_{A^d_p}^{\max\{1,\frac{1}{p-1}\}} \|f\|_{L^p(w)}.\]

In 1999, Nazarov, Treil and Volberg \cite{NTV1} found necessary and sufficient conditions on pairs $(u,v)$ of weights  for the uniform (on the choice of signs $\sigma$) boundedness of the martingale transforms $T_{\sigma}$ from $L^2(u)$ into $L^2(v)$. These conditions are a joint $A_2$-condition, a two weight Carleson condition and its dual condition and the two-weight boundedness of a certain positive operator. A few years later in 2006 they found necessary and sufficient conditions of Sawyer type for the boundedness of individual martingale transforms \cite{NTV2}.

It is well known, see \cite{P},  that when $\sigma_I=1$ for all $I\in\mathcal{D}$, $t=1$,  and $u=v\equiv 1$,  then $T^1_{w,1}$, denoted simply $T_w$,  is bounded from $L^p(\mathbb{R})$ into itself if and only if the weight $w$ is in the \emph{dyadic Reverse H\"older $p$ class} for $1<p<\infty$, namely
\begin{equation}
[w]_{RH^d_p}:=\sup_{I\in\mathcal{D}} \bigg(\frac{1}{|I|}\int_Iw(x)dx\bigg)^{-1}
\bigg(\frac{1}{|I|}\int_Iw^p(x)\, dx\bigg)^{\frac{1}{p}}<\infty. \label{def:RHp}
\end{equation}
Optimal quantitative estimates  on $L^p(\R)$ are known for $T_w$, see \cite{P2} for the quadratic bound on $L^2(\R )$  and  \cite[Chapter 4]{Pa} for $p\neq 2$. More precisely, there is a positive constant $C_p$ such that  for all $f\in L^p(\R)$ and for  $p'=p/(p-1)$ the dual exponent of $p$,
\[\|T_w f\|_{L^p(\R)} \leq C_p [w]_{RH^d_p}^{\max\{p,p'\}} \|f\|_{L^p(\R)}.\]
See \cite{P,P3} for how small variants of these operators first appeared naturally in the study of Sobolev spaces on Lipschitz curves and in connection to the resolvent of the dyadic paraproduct.

Note that the martingale transform is a Calder\'on-Zygmund operator but the variable Haar multipliers $T_w$ are not always Calder\'on-Zygmund type operators as they are not necessarily bounded on $L^p(\R )$ for all $1<p<\infty$. To see this, suffices to find a weight $w$ so that $w\in RH_p$ for $1<p<p_0$ but not for $p=p_0$, in that case $T_w$ will be bounded on $L^p(\R )$ if and only if $1<p<p_0$. The power weight $w_{\alpha}(x)=|x|^{\alpha}\in RH_p$ 
for $\alpha=-1/p_0$ has that property. These operators are outside the scope of the Aimar-Crescimbeni-Nowak theory of dyadic operators of Calder\'on-Zygmund type on spaces of homogeneous type \cite{ACN}.

The one- or two-weight theories for the $t$-Haar multipliers have not been studied before.
In this paper we find necessary and sufficient conditions on the triple $(u,v,w)$ for the uniform (on the choice of signs $\sigma$) boundedness of $T^t_{w,\sigma}$ from $L^2(u)$ into $L^2(v)$ that mimic the Nazarov-Treil-Volberg conditions for the martingale transform.   The conditions in the theorem are sufficient for individual $t$-Haar multipliers. We also find necessary and sufficient testing conditions  of Sawyer Type for individual $t$-Haar multipliers to be bounded from $L^2(u)$ into $L^2(v)$.


We prove the following two-weight theorem for $t$-Haar multipliers.  

\begin{theorem}[Two-weight theorem] \label{thm:two-weight} 
Given a triple of weights $(u,v,w)$ and $t\in \R$. Denote $\Delta_I w:= \langle w\rangle_{I_+} - \langle w\rangle_{I_-} $. The  $t$-Haar multipliers $T^t_{w,\sigma}$ are uniformly (on $\sigma$) bounded from $L^2(u)$ into $L^2(v)$ if and only if the following four conditions hold,
\begin{itemize}
\item[{\rm (i)}] Joint three-weight condition: $\displaystyle{\quad C_1:=\sup_{I\in\mathcal{D}} \frac{\langle u^{-1}\rangle_I \langle v w^{2t}\rangle_I}{\langle w\rangle_I^{2t}} <\infty}$. 
\item[{\rm (ii)}] Carleson condition: there is a constant $C_2>0$ such that for all $I\in\mathcal{D}$, 
$$\frac{1}{|I|} \sum_{J\in\mathcal{D}: J\subset I}  |J||\Delta_Ju^{-1}|^2 \frac{\langle vw^{2t}\rangle_J}{\langle w\rangle_J^{2t}}\leq C_2\,\langle u^{-1}\rangle_I.$$
\item[ {\rm (iii)}] Dual Carleson condition: there is a  constant $C_3>0$ such that for all $I\in\mathcal{D}$, 
$$\frac{1}{|I|} \sum_{J\in\mathcal{D}: J\subset I}  |J||\Delta_J (vw^{2t})|^2 \frac{\langle u^{-1}\rangle_J}{\langle w\rangle_J^{2t}}\leq C_3\, \langle vw^{2t} \rangle_I.$$
\item[{\rm (iv)}] The positive operator $P^t_{w,\lambda}$ is bounded from $L^2(u)$ into $L^2(v)$ with operator norm $C_4>0$, where  $\lambda_I:= 
\big ({|\Delta_I(vw^{2t})|}/{\langle vw^{2t}\rangle_I}\big )\big ({|\Delta_Iu^{-1}|}/{\langle u^{-1}\rangle_I} \big ) \big ({|I|}/{\langle w\rangle_I^t}\big )$ and 
$$P^t_{w,\lambda}f(x) := \sum_{I\in\mathcal{D}} \frac{w^t(x)}{|I|}\lambda_I \langle f\rangle_I \mathbbm{1}_I(x).$$
\end{itemize}
  Moreover  $ \|T^t_{w,\sigma} \|_{L^2(u)\to L^2(v)}\approx  \sqrt{C_1} + \sqrt{C_2} +\sqrt{C_3}+ C_4.$

\end{theorem}

%
When $\sigma\equiv 1$, the corresponding $t$-Haar multipliers are denoted by $T^t_w$ and were introduced in \cite{KP}. When $u=v\equiv 1$, Theorem~\ref{thm:two-weight} recovers known necessary and sufficient conditions for boundedness on $L^2(\R)$ when $t\leq 0$ or $t\geq 1/2$, and improves on the best sufficient condition that was known in the remaining cases $0<t<1/2$. The $t$-Haar multipliers $T^t_w$ were subsequently studied in \cite{P2} to get quantitative estimates on $L^2(\R)$ for the cases $t=1$, $t=\pm 1/2$ using the Bellman function technique. See also \cite[Chapter 5]{Be} for improved quantitative estimates for all $t$, \cite[Section 6.2]{Mo} and \cite{MoP} for further extensions of these results to $t$-Haar multipliers of higher multiplicity with bounds depending polynomially on the complexity. 

When restricting to the case $u=v$ and $t\notin (0,1)$, condition (iv)  in Theorem~\ref{thm:two-weight} is implied by conditions (i)-(iii), therefore we can prove the following one-weight theorem which is new in the literature.


\begin{theorem}[One-weight theorem] \label{thm:one-weight}
Given a pair of weights $(u,w)$, $t\leq 0$ or $t\geq 1$,  then the $t$-Haar multipliers $T^t_{w,\sigma}$ are uniformly (on $\sigma$) bounded from $L^2(u)$ into itself if and only if
\begin{itemize}
\item[{\rm (i)}] Joint two-weight condition: $\displaystyle{\quad C_1:=\sup_{I\in\mathcal{D}} \frac{\langle u^{-1}\rangle_I \langle u w^{2t}\rangle_I}{\langle w\rangle_I^{2t}} <\infty}$, 
\item[{\rm (ii)}] Carleson condition ($u^{-1}$-Carleson sequence): there is $C_2>0$  such that for all $I\in\mathcal{D}$,
$$\frac{1}{|I|} \sum_{J\in\mathcal{D}: J\subset I}  |J||\Delta_Ju^{-1}|^2 \frac{\langle uw^{2t}\rangle_J}{\langle w\rangle_J^{2t}}\leq C_2\, \langle u^{-1}\rangle_I.$$
\item[ {\rm (iii)}] Dual Carleson condition ($uw^{2t}$-Carleson sequence):  there is $C_3>0$  such that for all $I\in\mathcal{D}$,
$$\frac{1}{|I|} \sum_{J\in\mathcal{D}: J\subset I}  |J||\Delta_J (uw^{2t})|^2 \frac{\langle u^{-1}\rangle_J}{\langle w\rangle_J^{2t}}\leq C_3\, \langle uw^{2t} \rangle_I.$$

\end{itemize}
Moreover $\|T^t_{w,\sigma}\|_{L^2(u)\to L^2(u)} \approx \sqrt{C_1} + \sqrt{C_2} + \sqrt{C_3} + \sqrt{C_2C_3}.$

\end{theorem}

When $w\equiv 1$, Theorem~\ref{thm:one-weight} recovers the one-weight necessary and sufficient conditions for the boundedness of the martingale transform on $L^2(u)$,  namely $u\in A_2^d$, see \cite{TV,Wi}. In order to see that, note that when $w\equiv 1$, condition (i) in Theorem~\ref{thm:one-weight} becomes  $u\in A_2^d$. Conditions (ii) and (iii) are the well known Buckley summation conditions for the $RH_1^d$ weights $u$ and $u^{-1}$, see \cite{Bu} (see sections~\ref{sec:weights} and~\ref{sec:weight-prop} for definitions and details). 
Thus, when $w\equiv 1$, condition  (i) implies conditions (ii) and (iii), reducing the result to the known statement ``The martingale transform is uniformly (with respect to the choice of signs $\sigma$) bounded on $L^2(w)$ if and only if $w\in A_2^d$".

When $u\equiv 1$ and $t=1$, condition (i) in Theorem~\ref{thm:one-weight} becomes  $u\in RH_2^d$ and is equivalent to condition (iii) as we will prove in Proposition~\ref{prop:RHp} in Section~\ref{sec:weights}. Condition (ii) becomes trivially true  as $\Delta_J1=0$ for all $J\in \mathcal{D}$.  When in addition $\sigma\equiv 1$, this reduces to the known statement ``$T_w$ is bounded on $L^2(\R)$ if and only if $w\in RH_2^d$".

We  also obtain conditions for the boundedness of individual variable Haar multipliers in the spirit of the corresponding  Nazarov-Treil-Volberg results for individual  martingale transform \cite{NTV2}.

We have stated and proved all our theorems on $\R$.
These results are  valid  in the general set up where the underlying space is a space of homogeneous type (SHT), instead of $\mathbb{R}$. In order to be more precise, let us define such an operator. Let $(X,d,\mu)$ be a  space of homogeneous type with underlying  quasi-metric $d$ and doubling measure $\mu$, and $\mathfrak{D}$ be a dyadic system of cubes over it, the reader can find in  \cite[Section 3.5]{We} how to set up a dyadic system in these spaces so that each cube has at most two children and one Haar function associated to it. Such systems are sometimes called Wilson bases \cite{W}. Given a function $f \in L^2(X)$ and a weight $w: X \rightarrow \mathbb{R}$, we formally define the operator $T^t_{w,\sigma}$ as
$$
T^t_{w,\sigma}f(x):= \sum_{Q \in \mathfrak{D}} \sigma_Q \bigg( \frac{w(x)}{\langle w \rangle_Q}\bigg)^t \langle f, h_Q \rangle h_Q(x).
$$
Similarly, we could apply  conditions (i)-(iv) in Theorem~\ref{thm:two-weight}, or conditions (i)-(iii) in Theorem~\ref{thm:one-weight}, to the dyadic system $\mathfrak{D}$ and follow the same proof by extending the results used to a space of homogeneous type. Due to the chosen techniques for proving the result, we present the proof in $\mathbb{R}$ since it already encompasses all the key ideas. Writing the argument in the general SHT would only complicate the notation and unnecessarily lengthen the proof. The diligent reader can follow on the foot-steps of \cite{Ch1,We, NV,NRezV} to reproduce most of the results in $\R^n$ or in spaces of homogeneous type. Throughout the text we annotate where generalizations to  spaces of homogeneous type of the results used  can be found.

In Section~\ref{Defandlemma}, we introduce basic definitions and results we will use in the paper including weights, dyadic intervals,  dyadic $A_p$ and $RH_p$ classes of weights, weighted and unweighted Haar bases, weighted dyadic maximal function, weighted Carleson sequences and weighted Carleson's embedding theorem. In Section~\ref{sec:weight-prop}, we prove some propositions involving $RH^d_p$ and $A^d_p$ weights and  some summation conditions of Buckley type. In Section~\ref{sufficiency(i)-(iv)}, we prove that conditions (i)-(iv) in Theorem~\ref{thm:two-weight} are sufficient to obtain boundedness of each signed $t$-Haar multiplier, with a bound that is independent from the choice of signs $\sigma$. In Section~\ref{u=v=1}, we show these sufficient conditions recover  known necessary and sufficient conditions for the boundedness of $t$-Haar multipliers on $L^2(\R)$ when $\sigma=1$ and $t\leq 0$ or $t< 1/2$ (the case $u=v\equiv 1$). In the remaining case, $0<t\leq 1/2$, an improved sufficient condition is obtained that is also necessary.  In Section~\ref{u=v}, we show that in the one-weight case ($u=v$) corresponding conditions (i)-(iii) are sufficient for the uniform boundedness of the signed $t$-Haar multipliers on $L^2(u)$, provided $t\leq 0$ or $t\geq 1$ (Theorem~\ref{thm:one-weight}). In Section~\ref{sec:necessity(i)-(iii)}, we prove the necessity of conditions (i)-(iii) in Theorem~\ref{thm:two-weight}. In Section~\ref{sec:necessity(iv)}, after introducing the Nazarov, Treil and Volberg bilinear embedding theorem,  we show the necessity of condition (iv). Finally, in Section~\ref{individual}, we present necessary and sufficient testing conditions of Sawyer type for the boundedness of individual $t$-Haar multipliers.

The notation $f\equiv k$ means the function $f$ is identically equal to $k$. The notation $A\lesssim B$ means there exists a constant $C>0$ such that $A\leq CB$. The notation $A\approx B$ means that $A\lesssim B$ and $B\lesssim A$.




\subsubsection*{Acknowledgements}  The authors are grateful  to the anonymous referee for many comments that improved the presentation, in particular we appreciated the historical comments, setting the record straight.  B. D. Wick's research is supported in part by National Science Foundation Grants DMS \#1800057, \#2054863, and \#2000510 and Australian Research Council DP 220100285.  W. Huang acknowledges support from National Science Foundation Grant DMS \#1800057. D. Chung and J.C. Moraes acknowledge support from their institutions in South Korea and Brazil. M. C. Pereyra thanks them all for having kept some mathematics in her life during a particularly trying time for her family.

\section{Definitions and frequently used known results.}\label{Defandlemma}
In this section we introduce weights, dyadic intervals, dyadic $A_p$ and $RH_p$  classes of weights, the dyadic maximal function, the Haar and weighted Haar functions, and Carleson sequences. We also record known estimates for the dyadic maximal function  and the weighted Carleson's embedding lemma that will be used later.

\subsection{Weights and dyadic intervals}\label{sec:weights}
 A \emph{weight} $w$ is a non-negative locally integrable function in $\R$. The $w$-measure of a measurable set $E$ of $\R$ is denoted $w(E):= \int_Ew(x)\,dx$, and $|E|$ stands for the Lebesgue measure of $E$. We define $\langle f\rangle_E^w$ to be the integral average of $f$ on $E$ with respect to the measure $wdx$, 
 \[\langle f\rangle_E^w:=\frac{1}{w(E)}\int_E f(x) \,w(x)\,dx.\]
 When $w\equiv 1$ we simply write $\langle f\rangle_E$. 
 
 Given a weight $w$, a measurable function $f:\R\to \C$ is in $L^p(w)$ for $1\leq p <\infty$ if and only if $\|f\|_{L^p(w)}:=\big (\int_\R |f(x)|^p w(x)\,dx\big )^{\frac1p}<\infty$. In this paper we are concerned with the case $p=2$.

The \emph{standard dyadic intervals} on $\R$, denoted by $\mathcal{D}$, are the intervals of the form $[2^{-j}k, 2^{-j}(k+1))$ with $j,k\in\Z$. These intervals have the very useful nested property: if $I,J\in\mathcal{D}$ then one of the following holds: $I\subset J$, or $J\subset  I$, or $J\cap I=\emptyset$.  Given $I\in\mathcal{D}$ there is a unique interval $\tilde{I}\in \mathcal{D}$,  called the parent of $I$, such that $I\subset \tilde{I}$ and $|\tilde{I}|=2|I|$. Each interval $I\in \mathcal{D}$ has two ``children" in $\mathcal{D}$, namely its right and left halves, denoted  $I_{\pm}\in\mathcal{D}$. Note that $I=I_+\cup I_-$, $I_+\cap I_-=\emptyset$, and $|I_{\pm}|=|I|/2$.

We defined dyadic $A^d_p$ and $RH^d_p$ weights for $1<p<\infty$  in Section~\ref{sec:introduction} in displayed equations  \eqref{def:Ap} and \eqref{def:RHp}.   When $p=1$ we are also interested in the definition of the \emph{dyadic $RH^d_1$ class} of weights. A weight $w$ belongs to $RH^d_1$ if an only if
\begin{equation}\label{def:RH1}
[v]_{RH^d_1}:= \sup_{I\in\mathcal{D}}\left\langle \frac{v}{\langle v\rangle_I}\log \frac{v}{\langle v\rangle_I}\right\rangle_I < \infty.
\end{equation}
This is formally the limit as $p$ approaches $1^+$ of $[v]_{RH_p}$, see \cite{BeRez}.

Given a weight $w$ in $A^d_p$, its dual weight is defined to be $w^{-1/(p-1)}$.
 It is very simple to verify that the  dual weight $w^{-1/(p-1)}$ belongs to $A^d_{p'}$ with $1/p+1/p'=1$. Furthermore $[w]_{A_p^d}^{1/(p-1)}=[w^{-1/(p-1)}]_{A_{p'}^d}$.

All Muckenhoupt conditions $A_p^d$ and $RH_p^d$ in this note are dyadic conditions.
We define the $A^d_{\infty}$ \emph{class of weights} to be the union of all $A_p^d$ classes for $p>1$, that is $A_{\infty}^d=\cup_{p>1}A^d_p$. It can be shown that $RH^d_1$ is the union of all the $RH_p^d$ classes for $p>1$, that is $RH_1^d=\cup_{p>1}RH^d_p$. In the dyadic world
$A^d_{\infty}$ is a proper subset of $RH^d_1$, as $A^d_{\infty}$ implies dyadic doubling\footnote{A weight $w$ is \emph{dyadic doubling} weight if there is a constant $C>0$ such that $w(\tilde{I})\leq Cw(I)$ for all $I\in \mathcal{D}$ where $\tilde{I}$ is the parent of $I$.}, but $RH^d_1$ does not, see \cite{Bu}.  If we had not restricted the intervals to be dyadic we will have the ``continuous" or ``non-dyadic" $A_p$ and $RH_p$ classes and, in this case,  $A_{\infty}=RH_1$.

The theory of dyadic $A_p^d$ and $RH_p^d$ classes on spaces of homogeneous type can be found in~\cite{KaLPW}. Not everything translates verbatim, for example the celebrated G\"ehring's Lemma ($w\in RH_p$ implies $w\in RH_{p+\epsilon}$ for some $\epsilon>0$) is not true in the continuous case for  spaces of homogeneous type, however the property does hold for dyadic $RH_p^d$, see~\cite{AW} and references therein. Classical references for the theory of $A_p$ and $RH_p$ weights are \cite{GaRu, CrMPz}.

\subsection{Haar bases }
For any interval $I\subset \R$
there is a \emph{Haar function} defined by
$$h_I(x)=\frac{1}{\sqrt{|I|}}\Big(\mathbbm{1}_{I_+}(x)-\mathbbm{1}_{I_-}(x)\Big)\,,$$
where $\mathbbm{1}_I$ denotes the characteristic function of the interval $I\,$, and $I_+$, $I_-$ denote the right and left half of $I$ (the ``children" of $I$), respectively. For a given weight $v$ and an interval $I$, define the \emph{weighted Haar function} as
$$h_I^v(x)=\frac{1}{\sqrt{v(I)}}\left(\sqrt{\frac{v(I_-)}{v(I_+)}}
\,\mathbbm{1}_{I_+(x)}-\sqrt{\frac{v(I_+)}{v(I_-)}}\,\mathbbm{1}_{I_-(x)}\right)\,,$$
where as before $v(I)=\int_I v(x)\, dx$. Note that if $v\equiv 1$, then $h^1_I=h_I$.

The collections of Haar functions and of weighted Haar functions indexed on the dyadic intervals,  $\{h_I\}_{I\in\mathcal{D}}$
and $\{h^v_I\}_{I\in\mathcal{D}}$, form orthonormal systems on $L^2(\R )$ and $L^2(v)$ respectively. 
Therefore, for any weight $v$, by  Bessel's inequality we have the following:
$$\sum_{I\in\mathcal{D}}|\langle f,h_I^v\rangle_v|^2\leq \|f\|_{L^2(v)}^2\,,$$
where here $\langle f,g\rangle_v$ denotes the $L^2(v)$ inner product, namely $\int_{\R} f(x) \overline{g(x)} v(x)\, dx$.
\noindent  Moreover, if $v$ is a regular weight\footnote{A weight $v$ is \emph{regular} if and only if $v((-\infty,0))=v([0,\infty))=\infty$.}, then every square integrable function  $f\in L^2(v)$,
has an expansion in the corresponding weighted Haar basis $\{h^v_I\}_{I\in\mathcal{D}}$,
$$f(x)=\sum_{I\in\mathcal{D}}\langle f,h^v_I\rangle_v \,h^v_I(x)\,,$$
where the sum converges a.e. in $L^2(v)$. We conclude that for regular weights $v$,  the  family
$\{h_I^v\}_{I\in\mathcal{D}}$ is a complete orthonormal system. Note that if $v$ is
not a regular weight, so that $v((-\infty,0))$, $v([0,\infty))$, or both are finite,
then  either $\mathbbm{1}_{(-\infty,0)}$, $\mathbbm{1}_{(0,\infty)}$, or both are orthogonal to $h_I^v$ for every dyadic interval $I$. In particular the Haar functions indexed on the dyadic intervals form an orthonormal basis of $L^2(\R)$.  

For the dyadic system in $\mathbb{R}^n$, we recommend for the reader the paper \cite{W}, and for more details \cite[Section 2.2]{Ch1}. In spaces of homogeneous type (SHT), the basics can be found in \cite[Section 3.5]{We}. The proof that the Haar System defined on SHT $(X,d, \mu)$ forms a complete orthonormal basis for $L^2(X)$ can be found in \cite[Theorem 6.4.1]{We} or in \cite[Theorem 4.1]{KaLPW}.  

The weighted and unweighted Haar functions are related linearly as follows:

\begin{lemma}[\cite{NTV1}]\label{WHB}
For any weight $v$ and every $I\in\mathcal{D} $, there are numbers $\alpha_I^v$,
$\beta^v_I$ such that
$$ h_I(x) = \alpha^v_I \,h^v_I(x) + \beta_I^v \,\frac{\mathbbm{1}_I(x)}{\sqrt{|I|}}\,,$$

\noindent where \emph{(i)} $|\alpha^v_I | \leq \sqrt{\langle v\rangle_I},$
\emph{ (ii)}  $|\beta^v_I| \leq {|\Delta_I v|}/{\langle v\rangle_I},$
and $\Delta_I v:= \langle v\rangle_{I_+} - \langle v\rangle_{I_-}.$ 
\end{lemma}
In fact,   one can compute exactly the values of $\alpha^v_I$ and $\beta_I^v$ and deduce estimates (i) and (ii) from them. More precisely,
\begin{equation}\label{eqn:alpha-beta}
\alpha^v_I = \sqrt{{\langle v\rangle_{I_+}\langle v\rangle_{I_-}}/{\langle v\rangle_I}},      \quad\quad\quad\beta^v_I= {\Delta_Iv}/{2\langle v\rangle_I}.
\end{equation}

 For the generalization of Lemma~\ref{WHB} to $\mathbb{R}^n$ and to spaces of homogeneous type (SHT) see \cite[Section 2.2]{Ch2} and \cite[Proposition 8.3.9]{We}, respectively.

\subsection{Weighted Maximal function}

For a weight $u$ and a function $f \in L^{1}_{loc}(\mathbb{R})$, we define 
 the \emph{weighted dyadic maximal function of} $f$ by
\begin{equation}
M^d_uf(x):=\sup_{I \in \mathcal{D} \,: \,x\in I} \frac{1}{u(I)}\int_{I}|f(y)|u(y)\, dy=\sup_{I \in \mathcal{D} \, : \, x\in I}   \langle |f| \rangle^u_I. \label{wei_dyadic_maximal}
\end{equation}
 Here the supremum is taken over all intervals $I \in \mathcal{D}$ which contains the point $x$, and recall that $\langle  f\rangle^u_I$
denotes the weighted average of $f$ with respect to the weight $u$. Note that when $u\equiv 1$,  we obtain the dyadic maximal function, denoted  by $M^d$.
 
 The  $L^p(u)$ bound for the weighted dyadic maximal function  $M_u^d$ follows at once from the celebrated $L^p$ maximal inequality for martingales due to Doob, the reader can check his 1953 monograph \cite{Do}. One only needs to  pass from probability spaces to   $\R$, or more generally to spaces of homogeneous type.
 We record the result here as we will use these bounds repeatedly,  you can find a proof in \cite{CrMPz}.
\begin{lemma}\label{prop:weightedM}
Let $ 1<p< \infty$ and $u$ be a weight defined on $\mathbb{R}$. Then, for every function $f \in L^p(u)$,
\begin{equation*}
\| M^d_u f \|_{L^p(u)}  \leq p' \| f \|_{L^p(u)}.
\end{equation*}
\end{lemma}

The proof of this estimate in SHT   can be found in \cite{Ka, HKa}.


\subsection{Carleson sequences} \label{S:Def_Carl}
Given a weight $v$, a  sequence of positive real numbers $\{\lambda_I\}_{I\in\mathcal{D}}$ is a \emph{$v$-Carleson sequence} if and only if 
there is a constant $C>0$ such that 
\begin{equation}
\sum_{I\in\mathcal{D}(J)}\lambda_I\leq Cv(J), \quad\quad \mbox{for all  $J\in\mathcal{D}$}\,.\label{CarlS}
\end{equation}

\noindent When $v\equiv 1$ almost everywhere, we say that the sequence is a
\emph{Carleson sequence}. The infimum among all constants $C>0$  that satisfy inequality (\ref{CarlS}) 
is called the \emph{intensity} of the $v$-Carleson sequence $\{\lambda_I\}_{I\in\mathcal{D}}\,.$
For instance, if the reader is familiar with the space of functions of dyadic bounded mean oscillation $BMO^d$, a function $b\in BMO^d$ if and only if $\{|\langle b,h_I\rangle|^2\}_{I\in\mathcal{D}}$
is a Carleson sequence with intensity $\|b\|^2_{BMO^d}\,,$ for a proof see \cite{P}.

We now introduce a useful lemma which will be used frequently throughout  this paper. You can find a proof in \cite{MoP}, it was originally stated  in \cite{NTV1}.

\begin{lemma}[Weighted Carleson's Lemma] \label{WCL}
Let $v$ be a  weight, then $\{\alpha_{I}\}_{I \in \mathcal{D}}$ is a $v$-Carleson sequence
with intensity $\mathcal{B}$ if and only if for all non-negative $v$-measurable functions  $F$ on the
line,
\begin{equation}\label{eqn:WCL}
\sum_{I \in \mathcal{D}} (\inf_{x \in I} F(x) )\, \alpha_{I} \leq \mathcal{B}
\int_{\mathbb{R}}F(x) \,v(x)\,dx.
\end{equation}
\end{lemma}

For a proof of an analogous result in SHT, see  \cite[Lemma 8.3.5]{We}.

\begin{remark}\label{remark2} Let $g \in L^1_{loc}(u)$, if $F(x)= \big (M^d_{u}(g)(x)\big )^p$  for $1\leq p<\infty$  and $\{\alpha_I\}_{I\in\mathcal{D}}$ is a $u$-Carleson sequence with intensity $\mathcal{B}$, then by the Weighted Carleson's Lemma (Lemma~\ref{WCL})  we conclude that
\[\sum_{I\in\mathcal{D}}\alpha_I |\langle g\rangle_I^{u}|^p 
 \leq  \sum_{I\in\mathcal{D}}\alpha_I \inf_{x\in I} \big (M^d_{u}(g)(x)\big )^p 
 \leq  B \| M^d_{u}(g)\|_{L^p(u)}^p \; \leq  \; (p')^p\mathcal{B} \|g \|^p_{L^p(u)}.\]
Here, in the first inequality, we used that  $|\langle g\rangle^u_I|\leq \langle |g| \rangle_I^{u} \leq M^d_{u}(g)(x)$ for all $x\in I$, by definition of the weighted maximal function. In the last inequality, we applied Lemma~\ref{prop:weightedM}. This is how we will be using this lemma for various choices of weight $u$ and function $g$.
\end{remark}




\section{Some useful weight propositions}\label{sec:weight-prop}

In this section, we present some useful lemmas and propositions about weights.
When we specialize the packing or Carleson condition (iii) in Theorem~\ref{thm:one-weight} for $u \equiv 1$ we have a condition that is equivalent to $w\in RH^d_2$. Proposition~\ref{prop:RHp} will prove this, not only for $p=2$, but for the more general case $p>1$. Moreover, when $w\equiv 1$ and $u=v$, Proposition~\ref{prop:Ap} will show that $u^{-1}\in A_2$ (if and only if $u\in A_2$) implies the dual Carleson condition  (iii). Also, Proposition~\ref{prop:Ap} proves that $u\in A_2^d$ implies the packing or Carleson condition (ii) in Theorem~\ref{thm:one-weight}.

\begin{proposition}\label{prop:RHp}  Let $1<p<\infty$. A weight $w$ lies in $RH^d_p$ if and only if there is a constant $C>0$ such that
\begin{equation}\label{eqn:RHp}
 \frac{1}{|I|} \sum_{J\in\mathcal{D}: J\subset I} |J| \frac{|\Delta_J(w^p)|^2}{\langle w\rangle_J^p} \leq C \langle w^p\rangle_I  \quad \mbox{for all $I\in\mathcal{D}$}, 
 \end{equation}
where we recall that $\Delta_Jw= \langle w\rangle_{J_+}-\langle w\rangle_{J_-}$.
\end{proposition}

To prove Proposition~\ref{prop:RHp} we will need two auxiliary  lemmas.

\begin{lemma}[Buckley's characterization of $RH^d_1$ \cite{Bu, BeRez}]\label{lem:BuckleyRH1} A weight $v$ lies in $RH^d_1$ if and only if there is a constant $C>0$ such that
\begin{equation}\label{eqn:BuckleyRH1}
\frac{1}{|I|} \sum_{J\in\mathcal{D}: J\subset I} |J| \frac{|\Delta_J v|^2}{\langle v\rangle_J} \leq C[v]_{RH^d_1}\langle v\rangle_I  \quad \mbox{for all $I\in\mathcal{D}$}.
\end{equation}
\end{lemma}

The quantitative result  (the fact that the constant on the right-hand-side is comparable to $[v]_{RH_1}$) can be found in \cite[Equation (II.2) in Theorem II.1]{BeRez}, see  \cite{Bu} for the qualitative result.

\begin{lemma}[Katz-Pereyra `99, Lemma 1]\label{lem:KP-RHp} let $1<p<\infty$. 
A weight $w$ lies in $RH^d_p$ if and only if $w^p$ lies in $RH^d_1$.
\end{lemma}

Armed with these two results we can prove Proposition~\ref{prop:RHp}.

\begin{proof}[Proof of Proposition~\ref{prop:RHp}]
($\Rightarrow$) Assume $w\in RH^d_p$,  hence by Lemma~\ref{lem:KP-RHp} we conclude that $w^p\in RH^d_1$. By the definition of $RH^d_p$ we have that  $\langle w^p\rangle_I \langle w\rangle_I^{-p} \leq [w]_{RH^d_p}^p$. Hence multiplying and dividing by $\langle w^p\rangle_J$, we get that
\begin{eqnarray*}
\frac{1}{|I|} \sum_{J\in\mathcal{D}: J\subset I} |J| \frac{|\Delta_J(w^p)|^2}{\langle w\rangle_J^p} & \leq & [w]_{RH^d_p}^p \frac{1}{|I|} \sum_{J\in\mathcal{D}: J\subset I} |J| \frac{|\Delta_J(w^p)|^2}{\langle w^p\rangle_J}  \\& \leq & C[w]_{RH^d_p}^p [w^p]_{RH^d_1} \langle w^p\rangle_I,
\end{eqnarray*}
where the last inequality comes from Lemma~\ref{lem:BuckleyRH1} applied to $v=w^p\in RH^d_1$.

($\Leftarrow$) Assume now that the  packing condition \eqref{eqn:RHp} holds, then we will show that $w^p$ satisfies Buckley's $RH^d_1$ packing  condition \eqref{eqn:BuckleyRH1}. Finally,  we will use Lemma~\ref{lem:KP-RHp} to conclude that  $w\in RH^d_p$.  Indeed, 
\[ \frac{1}{|I|} \sum_{J\in\mathcal{D}: J\subset I} |J| \frac{|\Delta_J(w^p)|^2}{\langle w^p\rangle_J} \leq 
 \frac{1}{|I|} \sum_{J\in\mathcal{D}: J\subset I} |J| \frac{|\Delta_J(w^p)|^2}{\langle w\rangle_J^p}  \leq  C \langle w^p\rangle_I. \]
Here the first estimate holds since, by H\"older's inequality,  $\langle w^p\rangle_J^{-1} \leq  \langle w\rangle_J^{-p}$.  
\end{proof}

%

For this proposition to hold in spaces of homogeneous type (SHT), we will need Buckley's characterization of $RH_1^d$ and the Katz-Pereyra Lemma to hold in SHT.
The Katz-Pereyra Lemma uses G\"ehring's self-improvement property of $RH_p^d$, which does hold in SHT provided the spaces support a Lebesgue Differentiation Theorem \cite{AW}, the rest of the proof goes through in SHT. We have not been able to find a reference  for Buckley's characterization of $RH_1^d$ in SHT. That is an interesting question.
\begin{question} Is there an analogue of Buckley' characterization of $RH^d_1$ in spaces of homogeneous type?
\end{question}

We can prove an  $A^d_p$ version of Proposition~\ref{prop:RHp}.  First we state an prove the following  $A^d_p$ version of the Katz-Pereyra Lemma, that will suffice for our purposes.

\begin{lemma}\label{lem:KP-Ap}
If a weight $w$ lies in $A^d_p$ then $w^{\frac{-1}{p-1}}$ lies in  $RH^d_1$.  
\end{lemma}


\begin{proof}[Proof of Lemma~\ref{lem:KP-Ap}]
%
Given a weight  $w$ in $A^d_p$, we have that the dual weight $w^{\frac{-1}{p-1}}$ is in $A^d_{p'}$, with $\frac{1}{p}+\frac{1}{p'}=1$.
 Since $RH^d_1$ contains $A^d_{\infty}$, and 
 $A^d_{\infty}=\cup_{q>1}A^d_q$, then $RH^d_1$ contains $A^d_q$ for every $q>1$. In particular, $A^d_{p'}$ is contained in $RH^d_1$.
We conclude that  $w^{\frac{-1}{p-1}}$ lies in $RH^d_1$, as claimed.
\end{proof}
\begin{proposition}\label{prop:Ap}  If a weight $w$ lies in $A^d_p$ for $p>1$, then there is a constant $C>0$ such that
\begin{equation}\label{eqn:Carleson}  \frac{1}{|I|} \sum_{J\in\mathcal{D}: J\subset I} |J| \frac{|\Delta_J(w^{\frac{-1}{p-1}})|^2}{\langle w\rangle_J^{\frac{-1}{p-1}}} \leq C \langle w^{\frac{-1}{p-1}}\rangle_I,  \quad \mbox{for all $I\in \mathcal{D}$}.
\end{equation}
Moreover a partial converse holds, i.e., if the condition \eqref{eqn:Carleson} holds then $w^{\frac{-1}{p-1}}$ lies in $RH^d_1$ with $[w^{\frac{-1}{p-1}}]_{RH^d_1}\sim C$.
\end{proposition}


\begin{proof}[Proof of Proposition~\ref{prop:Ap}]
Assume $w\in A^d_p$, then $\langle w\rangle_J \langle w^{\frac{-1}{p-1}}\rangle_J^{p-1} \leq [w]_{A^d_p}$  for all dyadic intervals $J$, and by  Lemma~\ref{lem:KP-Ap}, $w^{\frac{-1}{p-1}}\in RH^d_1$. Hence multiplying and dividing by $\langle w^{\frac{-1}{p-1}}\rangle_J$, we get,
\begin{eqnarray*}
 \frac{1}{|I|} \sum_{J\in\mathcal{D}: J\subset I} |J| \frac{|\Delta_J(w^{\frac{-1}{p-1}})|^2}{\langle w\rangle_J^{\frac{-1}{p-1}}} 
 &\leq & [w]_{A^d_p}^{\frac{1}{p-1}} \frac{1}{|I|} \sum_{J\in\mathcal{D}: J\subset I} |J| \frac{|\Delta_J(w^{\frac{-1}{p-1}})|^2}{\langle w^{\frac{-1}{p-1}}\rangle_J}  \\
 &\leq & C[w]_{A^d_p}^{\frac{1}{p-1}} [w^\frac{-1}{p-1}]_{RH^d_1}\langle w^{\frac{-1}{p-1}}\rangle_I.
 \end{eqnarray*}
Here the last inequality comes from Lemma~\ref{lem:BuckleyRH1} applied to $v=w^{\frac{-1}{p-1}}\in RH^d_1$.

Assume now that  the Carleson condition~\eqref{eqn:Carleson} holds, then we will show that $w^{\frac{-1}{p-1}}$ satisfies Buckley's $RH^d_1$ characterization condition. 
Indeed, 
\[ \frac{1}{|I|} \sum_{J\in\mathcal{D}: J\subset I} |J| \frac{|\Delta_J(w^{\frac{-1}{p-1}})|^2}{\langle w^{\frac{-1}{p-1}}\rangle_J}  \leq 
 \frac{1}{|I|} \sum_{J\in\mathcal{D}: J\subset I} |J| \frac{|\Delta_J(w^{\frac{-1}{p-1}})|^2}{\langle w\rangle_J^{\frac{-1}{p-1}}}  \leq  C \langle w^{\frac{-1}{p-1}}\rangle_I. \]
 Where the first estimate holds  since, by H\"older's inequality,  $\langle w^{\frac{-1}{p-1}}\rangle_J^{-1} \leq  \langle w\rangle_J^{\frac{1}{p-1}}$.   The second estimate by hypothesis~\eqref{eqn:Carleson}. By Lemma~\ref{lem:BuckleyRH1}, we conclude that $w^{\frac{-1}{p-1}}$ lies in $RH_1^d$. 
\end{proof}
To prove Proposition~\ref{prop:Ap} in spaces of homogeneous type (SHT) we will also need Buckley's characterization of $RH_1^d$ in SHT.  The main theorems proved in this paper will not need Proposition~\ref{prop:RHp} nor Proposition~\ref{prop:Ap}.

Note that if $w\in A_p^d$ then the dual weight $w^{-\frac{1}{p-1}}\in A_{p'}^d$. If we apply Proposition~\ref{prop:Ap} to the dual weight with $p$ replaced by $p'$, we conclude that for all $w\in A_p^d$,
\begin{equation}\label{eqn:dualCarleson}  \frac{1}{|I|} \sum_{J\in\mathcal{D}: J\subset I} |J| {|\Delta_J(w)|^2}{\langle w^{\frac{-1}{p-1}}\rangle_J^{p-1}} \leq C \langle w\rangle_I,  \quad \mbox{for all $I\in \mathcal{D}$}.
\end{equation}


%
%


\section{Sufficient conditions for two-weight estimates for the $t$-Haar Multiplier $T^t_{w,\sigma}$}\label{sufficient-conditions}

We are now able to start  proving Theorem~\ref{thm:two-weight}. In Section~\ref{sufficiency(i)-(iv)}, we will prove that conditions (i)-(iv) are sufficient for each $T^t_{w,\sigma}$ to be uniformly (on $\sigma$)  bounded from $L^2(u)$ into $L^2(v)$. In fact, for any choice of $\sigma$, we will prove in Theorem~\ref{thm:2weight-t-HaarMultiplier}, that conditions (i)-(iv)  guarantee the boundedness of the individual operator $T^t_{w,\sigma}$ from $L^2(u)$ into $L^2(v)$,  and the estimate does not depend on $\sigma$. We will need to assume the uniform boundedness (with respect to $\sigma$) of the $t$-Haar multipliers $T^t_{w,\sigma}$  when proving  the necessity of conditions (i)-(iv) in Section~\ref{sec:necessary}.


In Section~\ref{u=v=1}, we show that in the unweighted and unsigned case, that is  when $u=v\equiv 1$ and $\sigma\equiv 1$, we recover all known necessary and sufficient conditions for the boundedness of $T^t_w$  when $t\leq 0$ or when $t\geq 1/2$. We improve upon the known sufficient condition when $0<t<1/2$.

In Section~\ref{u=v}, we show that in the one weight case, that is when $u=v$, conditions (i)-(iii) imply condition (iv) when $t\leq 0$ or when $t\geq 1$. Hence we obtain a new one weight theorem that says that conditions (i)-(iii) are sufficient for the uniform (with respect to the choice of signs $\sigma$) boundedness of $T^t_{w, \sigma}$ on $L^2(u)$
when $t\leq 0$ or $t\geq 1$. Later on it will be clear these are also necessary conditions.
\subsection{The sufficiency of conditions (i)-(iv)}\label{sufficiency(i)-(iv)}
In this section we prove that conditions (i)-(iv) in Theorem~\ref{thm:2weight-t-HaarMultiplier} are sufficient for the boundedness of the $t$-Haar multiplier  $T^t_{w,\sigma}$ for each fixed $t\in \R$, $w$ a weight, and $\sigma$ a choice of signs.

\begin{theorem}[Two-weight theorem: sufficient conditions]\label{thm:2weight-t-HaarMultiplier}
Given $t\in \R$, $(u,v,w)$ a triple of weights, and a sequence of signs $\sigma$. 
Denote $\Delta_I w:= \langle w\rangle_{I_+} - \langle w\rangle_{I_-} $. The $t$-Haar multiplier $T_{w,\sigma}^t$ is bounded from $L^2(u)$ into $L^2(v)$ if the following hold:

\begin{itemize}
\item[{\rm (i)}]   A three-weight condition:  $\displaystyle{\;\; C_1:=\sup_{I\in\mathcal{D}} \frac {\langle u^{-1}\rangle_I \langle vw^{2t}\rangle_I}{\langle w\rangle_I^{2t}} <\infty. }$   

\item[{\rm (ii)}] A  three weight Carleson condition:  for all $I\in \mathcal{D}$ there is a constant $C_2>0$ such that
\[ \frac{1}{|I|}  \sum_{J\in\mathcal{D}: J\subset I} |J| \frac{\langle vw^{2t}\rangle_J |\Delta_J (u^{-1})|^2}{\langle w\rangle_J^{2t}} \leq C_2\, \langle u^{-1}\ \rangle_I.\]
In other words,  $\left \{ \mu_J:=|J| \frac{\langle vw^{2t}\rangle_J |\Delta_J (u^{-1})|^2}{\langle w\rangle_J^{2t}} \right \}_{J\in\mathcal{D}}$ is a $u^{-1}$-Carleson sequence.
\item[{\rm (iii)}] A dual three weight Carleson condition: for all $I\in \mathcal{D}$ there is a constant $C_3>0$ such that
\[ \frac{1}{|I|}  \sum_{J\in\mathcal{D}: J\subset I} |J| \frac{\langle u^{-1}\rangle_J |\Delta_J (vw^{2t})|^2}{\langle w\rangle_J^{2t}} \leq C_3\, \langle vw^{2t} \rangle_I.\]
In other words, $\left \{\rho_J:= |J| \frac{\langle u^{-1}\rangle_J |\Delta_J (vw^{2t})|^2}{\langle w\rangle_J^{2t}}\right \}_{J\in\mathcal{D}}$ is a $vw^{2t}$-Carleson sequence.
\item[{\rm (iv)}] The positive operator $P^t_{w,\lambda}$ is bounded from $L^2(u)$ into $L^2(v)$ with operator norm $C_4>0$, where
\[ P^t_{w,\lambda} f(x) := \sum_{I\in\mathcal{D}} \frac{w^t(x)}{|I|} \lambda_I \langle f\rangle_I \mathbbm{1}_I(x), \quad
\mbox{and}\;\; \lambda_I= \frac{|\Delta_I(u^{-1})|}{\langle u^{-1}\rangle_I}\frac{|\Delta_I (vw^{2t})|}{\langle vw^{2t}\rangle_I} \frac{|I|}{\langle w\rangle_I^t}.\] 
\end{itemize}
Moreover $\|T^t_{w,\sigma}\|_{L^2(u)\to L^2(v)} \lesssim \sqrt{C_1} + \sqrt{C_2} + \sqrt{C_3}+ C_4$.
\end{theorem}

\begin{proof}

\noindent By duality it suffices to prove that
$$|\langle T^t_{w,\sigma}(fu^{-1}), gv\rangle |\lesssim (\sqrt{C_1}+\sqrt{C_2}+\sqrt{C_3}+C_4)  \|f\|_{L^2(u^{-1})}\|g\|_{L^2(v)}.$$
Note that $fu^{-1}\in L^2(u)$ if and only if $f\in L^2(u^{-1})$, and similarly, $gv\in L^2(v^{-1})$ if and only if $g\in L^2(v)$, with equal norms in both cases. Also note that $L^2(v^{-1})$ is the dual of the space $L^2(v)$ under the $L^2$-pairing.

First observe that,
\begin{eqnarray} \langle T^t_{w,\sigma}(fu^{-1}), gv\rangle  & = & \left\langle \sum_{I\in\D} \sigma_I \frac{w^t}{\langle w\rangle_I ^t}\langle u^{-1}f,h_I\rangle h_I, gv \right\rangle \nonumber \\
&=& \sum_{I\in\D}  \frac{\sigma_I}{\langle w\rangle_I ^t}\langle u^{-1}f,h_I\rangle \langle  h_I, gvw^t\rangle. \label{t1}
\end{eqnarray} 

Here we have exchanged  inner product and summation. If the operator is bounded then we have convergence in $L^2(v)$, and the exchange is legal by the continuity of the inner product. However we are trying to show the operator is bounded. The calculations we are about to perform hold true for finite sums and the estimates are independent of the truncation parameters. At the end of the day, we will be able to take limits on the truncation parameters and deduce the boundedness of our operators from $L^2(u)$ into $L^2(v)$, given conditions (i)-(iv). We will let the reader fill in the details.


We will use the linear relation between  the weighted and unweighted Haar functions provided by Lemma~\ref{WHB} for weights $u^{-1}$ and $vw^{2t}$:
\begin{equation}
h_I(x)=\alpha_I^{u^{-1}} h_I^{u^{-1}}(x)+\beta_I^{u^{-1}}\frac{\mathbbm{1}_I(x)}{\sqrt{|I|}}\,, \label{rel:unbalaced_haar_u}
\end{equation}
where $|\alpha_I^{u^{-1}}|\leq \sqrt{\langle u^{-1}\rangle_I}$ and $|\beta_I^{u^{-1}}|\leq \frac{|\Delta_I (u^{-1})|}{\langle u^{-1}\rangle_I}$, and
\begin{equation}
h_I(x)=\alpha_I^{vw^{2t}} h_I^{vw^{2t}}(x)+\beta_I^{vw^{2t}}\frac{\mathbbm{1}_I(x)}{\sqrt{|I|}}\,, \label{rel:unbalaced_haar_vw2}
\end{equation}
where $|\alpha_I^{vw^{2t}}|\leq \sqrt{\langle vw^{2t}\rangle_I}$ and $|\beta_I^{vw^{2t}}|\leq \frac{|\Delta_I (vw^{2t})|}{\langle vw^{2t}\rangle_I}$.

Replace $h_I$ in (\ref{t1}) by (\ref{rel:unbalaced_haar_u}) and (\ref{rel:unbalaced_haar_vw2}) in the inner product involving $f$ and $g$ respectively, then  use the distributive property to break the sum into four sums:
\begin{equation}\label{eqn:decompositon4terma}
\langle T^t_{w,\sigma}(fu^{-1}), gv\rangle =   \Sigma_1(f,g) + \Sigma_2(f,g) + \Sigma_3(f,g) + \Sigma_4(f,g),
\end{equation}
where
\[ \Sigma_1(f,g)=\sum_{I\in\D}\frac{\sigma_I}{\langle w\rangle_I^t}\left\langle u^{-1}f, \alpha_I^{u^{-1}} h_I^{u^{-1}}\right\rangle \left\langle\alpha_I^{vw^{2t}} h_I^{vw^{2t}}, gvw^t \right\rangle ,\]

\[\Sigma_2(f,g)=\sum_{I\in\D}\frac{\sigma_I}{\langle w\rangle_I^t}\left\langle u^{-1}f, \beta_I^{u^{-1}}\frac{\mathbbm{1}_I}{\sqrt{|I|}}\right\rangle \left\langle \alpha_I^{vw^{2t}} h_I^{vw^{2t}}, gvw^t \right\rangle,\]

\[ \Sigma_3(f,g)=\sum_{I\in\D}\frac{\sigma_I}{\langle w\rangle_I^t}\left\langle u^{-1}f, \alpha_I^{u^{-1}}h_I^{u^{-1}}\right\rangle \left\langle \beta_I^{vw^{2t}}\frac{\mathbbm{1}_I}{\sqrt{|I|}}, gvw^t\right\rangle,\]

\[\Sigma_4(f,g)=\sum_{I\in\D}\frac{\sigma_I}{\langle w\rangle_I^t}\left\langle u^{-1}f, \beta_I^{u^{-1}}\frac{\mathbbm{1}_I}{\sqrt{|I|}}\right\rangle \left\langle \beta_I^{vw^{2t}}\frac{\mathbbm{1}_I}{\sqrt{|I|}}, gvw^t\right\rangle.\]

By the triangle inequality 
\[ |\langle T^t_{w,\sigma}(fu^{-1}), gv\rangle | \leq |\Sigma_1(f,g)| + |\Sigma_2(f,g)| + |\Sigma_3(f,g)| + |\Sigma_4(f,g)|.\]

%
%
%
%
%

We now estimate the terms separately. Each of the hypothesis: (i), (ii), (iii), and (iv), will be needed to prove the estimate for the corresponding terms: $\Sigma_1(f,g)$, $\Sigma_2(f,g)$, $\Sigma_3(f,g)$, and $\Sigma_4(f,g)$. Each will contribute to the operator norm $\sqrt{C_1}$, $\sqrt{C_2}$, $\sqrt{C_3}$, and $C_4$,  respectively.

 First, let us estimate $\Sigma_1(f,g)$.   Using the triangle inequality, the bounds on  $\alpha^{u^{-1}}_I$ and $\alpha_I^{vw^{2t}}$,  and the fact that $\langle uf,g\rangle=\langle f,g\rangle_u$, we have that,
\begin{align*}
 |\Sigma_1(f,g)| &\leq \sum_{I\in\D}\frac{1}{\langle w\rangle_I^t}\left|\left\langle u^{-1}f, \alpha_I^{u^{-1}} h_I^{u^{-1}}\right\rangle\right| \left|\left\langle\alpha_I^{vw^{2t}} h_I^{vw^{2t}}, gvw^t \right\rangle \right|\\
 &\leq \sum_{I\in\D}\frac{\sqrt{\langle u^{-1}\rangle_I}}{\langle w\rangle_I^t}\left|\left\langle u^{-1}f,h_I^{u^{-1}}\right\rangle\right|\sqrt{\langle vw^{2t}\rangle_I}\left|\left\langle h_I^{vw^{2t}},g w^{-t}vw^{2t}\right\rangle\right|\\
 &\leq \sum_{I\in\D}\frac{\sqrt{\langle u^{-1}\rangle_I}\sqrt{\langle vw^{2t}\rangle_I}}{\langle w\rangle_I^t}\left|\left\langle f,h_I^{u^{-1}}\right\rangle_{u^{-1}}\right|\left|\left\langle h_I^{vw^{2t}},g w^{-t}\right\rangle_{vw^{2t}}\right|.
\end{align*}
Now, we are able to use  condition (i), followed by the Cauchy-Schwarz inequality, and the fact that the $\{h_I^{u^{-1}}\}_{I\in\mathcal{D}}$ and $\{h_I^{vw^{2t}}\}_{I\in\mathcal{D}}$ are orthonormal systems  in $L^2(u^{-1})$  and $L^2(vw^{2t})$, to obtain,
\begin{align*}
|\Sigma_1(f,g)|  &\leq \sqrt{C_1}\,\left(\sum_{I\in\D}\left|\left\langle f,h_I^{u^{-1}}\right\rangle_{u^{-1}}\right|^2\right)^{1/2}\left(\sum_{I\in\D}\left|\left\langle h_I^{vw^{2t}},g w^{-t}\right\rangle_{vw^{2t}}\right|^2\right)^{1/2}\\
 &\leq \sqrt{C_1}\,\|f\|_{L^2(u^{-1})}\|gw^{-t}\|_{L^2(vw^{2t})}\\
 &=\sqrt{C_1}\,\|f\|_{L^2(u^{-1})}\|g\|_{L^2(v).}
\end{align*}
 
Second, let us estimate $\Sigma_2(f,g)$. Using the triangle inequality, the bounds on  $\beta^{u^{-1}}_I$ and $\alpha_I^{vw^{2t}}$, and the facts that $\langle uf,g\rangle=\langle f,g\rangle_u$, $\langle h,{\mathbbm{1}_I}/{\sqrt{|I|}}\rangle =  \sqrt{|I|}\langle h\rangle_I$, and $\langle hu\rangle_I=\langle h\rangle^u_I \langle u\rangle_I$, we have that, 
\begin{align*}
|\Sigma_2(f,g)|&=\sum_{I\in\D}\frac{1}{\langle w\rangle_I^t}\left|\left\langle u^{-1}f, \beta_I^{u^{-1}}\frac{\mathbbm{1}_I}{\sqrt{|I|}}\right\rangle\right|\left|\left\langle \alpha_I^{vw^{2t}} h_I^{vw^{2t}}, gvw^t \right\rangle\right|\\
&\leq \sum_{I\in\D}\frac{1}{\langle w\rangle_I^t}\frac{|\Delta_I (u^{-1})|\sqrt{|I|}}{\langle u^{-1}\rangle_I }\sqrt{\langle vw^{2t}\rangle_I}\;
\langle |u^{-1}f|\rangle_I\left|\left\langle gw^{-t}, h_I^{vw^{2t}}\right\rangle_{vw^{2t}}\right|\\
&\leq \sum_{I\in\D}\frac{1}{\langle w\rangle_I^t}|\Delta_I (u^{-1})|\sqrt{|I|}\sqrt{\langle vw^{2t}\rangle_I}\;
\langle |f|\rangle_I^{u^{-1}}\left|\left\langle gw^{-t}, h_I^{vw^{2t}}\right\rangle_{vw^{2t}}\right|.
\end{align*}
Next apply the Cauchy-Schwarz inequality, and we are set to use condition (ii)  and the Weighted Carleson Lemma (Lemma 4) with the $u^{-1}$-Carleson sequence $\mu_I=|I|\,{|\Delta_I u^{-1}|^2 \;\langle vw^{2t}\rangle_I }/{\langle w\rangle_I^{2t}}$. Finally, use estimates for the weighted maximal function (Proposition~\ref{prop:weightedM}) to obtain the desired bound,
\begin{align*}
|\Sigma_2(f,g)|&\leq \left(\sum_{I\in\D}\mu_I \inf_{x\in I} M_{u^{-1}}^2f(x)\right)^{1/2}\left(\sum_{I\in\D}\left|\left\langle gw^{-t}, h_I^{vw^{2t}}\right\rangle_{vw^{2t}}\right|^2\right)^{1/2}\\
&\leq \sqrt{C_2} \,\left(\int M^2_{u^{-1}}f(x)\,u^{-1}(x)\,dx\right)^{1/2}\|gw^{-t}\|_{L^2(vw^{2t})}\\
&\leq 2\sqrt{C_2} \,\|f\|_{L^2(u^{-1})}\|g\|_{L^2(v)}.
\end{align*}

Third, similarly let us estimate $\Sigma_3(f,g)$. Using the triangle inequality, the bounds on  $\alpha^{u^{-1}}_I$ and $\beta_I^{vw^{2t}}$, and the facts that $\langle uf,g\rangle=\langle f,g\rangle_u$, and $\langle hu\rangle_I=\langle h\rangle^u_I \langle u\rangle_I$, we get,
\begin{align*}
|\Sigma_3(f,g)|&=  \sum_{I\in\D}\frac{1}{\langle w\rangle_I^t}\left|\left\langle u^{-1}f, \alpha_I^{u^{-1}}h_I^{u^{-1}}\right\rangle\right|\left|\left\langle \beta_I^{vw^{2t}}\frac{\mathbbm{1}_I}{\sqrt{|I|}}, gvw^t \right\rangle\right|\\
&\leq \sum_{I\in\D}\frac{\sqrt{\langle u^{-1}\rangle_I}}{\langle w\rangle_I^t}\left|\left\langle u^{-1}f,h_I^{u^{-1}}\right\rangle\right|\frac{|\Delta_I (vw^{2t})|}{\langle vw^{2t}\rangle_I}\sqrt{|I|} \, \langle |g|vw^t\rangle_I\\
 &= \sum_{I\in\D}\left|\left\langle f,h_I^{u^{-1}}\right\rangle_{u^{-1}}\right|\frac{\sqrt{\langle u^{-1}\rangle_I}}{\langle w\rangle_I^t}\frac{|\Delta_I (vw^{2t})|}{\langle vw^{2t}\rangle_I}\sqrt{|I|} \, \langle |g|vw^t\rangle_I\\
&=\sum_{I\in\D}\left|\left\langle f,h_I^{u^{-1}}\right\rangle_{u^{-1}}\right|\frac{\sqrt{\langle u^{-1}\rangle_I}\,|\Delta_I (vw^{2t})|\sqrt{|I|}}{\langle w\rangle_I^t} \langle |g|w^{-t}\rangle_I^{vw^{2t}}.
\end{align*}
Now apply the Cauchy-Schwarz inequality, and we are set to use condition (iii) and  the Weighted Carleson Lemma (Lemma~\ref{WCL}) with the $vw^{2t}$-Carleson sequence $\rho_I= |I|\,{\langle u^{-1}\rangle_I\,|\Delta_I (vw^t)|^2}/{\langle w\rangle_I^{2t}}$, together with  estimates for the weighted maximal function (Lemma~\ref{prop:weightedM}). We have that,
\begin{align*}
|\Sigma_3(f,g)| &\leq \left(\sum_{I\in\D}\left|\left\langle f,h_I^{u^{-1}}\right\rangle_{u^{-1}}\right|^2\right)^{1/2}\left(\sum_{I\in\D}\rho_I\inf_{x\in I}M^2_{vw^{2t}}(gw^{-t})(x)\right)^{1/2}\\
&\leq \sqrt{C_3}\, \|f\|_{L^2(u^{-1})}\left(\int_{\R} M_{vw^{2t}}^2 (gw^{-t})(x)\, v(x)w^{2t}(x)\,dx\right)^{1/2}\\
&\leq 2 \sqrt{C_3} \,\|f\|_{L^2(u^{-1})}\|gw^{-t}\|_{L^2(vw^{2t})}\\
&=2 \sqrt{C_3}\, \|f\|_{L^2(u^{-1})}\|g\|_{L^2(v)}.
\end{align*}

Fourth, in order to estimate $\Sigma_4(f,g)$, we will  use the triangle inequality, and the bounds on $\beta^{u^{-1}}_I$ and $\beta^{vw^{2t}}_I$, to get,
\begin{align*}
|\Sigma_4(f,g)| &=\sum_{I\in\D}\frac{1}{\langle w\rangle_I^t}\left|\left\langle u^{-1}f, \beta_I^{u^{-1}}\frac{\mathbbm{1}_I}{\sqrt{|I|}}\right\rangle\right|\left|\left\langle \beta_I^{vw^{2t}}\frac{\mathbbm{1}_I}{\sqrt{|I|}}, gvw^t\right\rangle\right|\\  
&\leq \sum_{I\in\D}\frac{1}{\langle w\rangle_I^t}\frac{|\Delta_I (u^{-1})|\sqrt{|I|}}{\langle u^{-1}\rangle_I} \langle |u^{-1}f| \rangle_I
\frac{|\Delta_I (vw^{2t})|\sqrt{|I|}}{\langle vw^{2t}\rangle_I} \langle |gvw^t|\rangle_I\\
&=\sum_{I\in\D}\lambda_I \, \langle |f|u^{-1}\rangle_I \,  \langle |g|vw^t\rangle_I\\
&=\sum_{I\in\D}\lambda_I \,\langle u^{-\frac12}(|f|u^{-{\frac12}})\rangle_I\, \langle v^{\frac12}|g|v^{\frac12}w^t\rangle_I,
\end{align*}
where $\lambda_I=({|I|}/{\langle w\rangle_I^t})({|\Delta_I(u^{-1})|}/{\langle u^{-1}\rangle_I})({|\Delta_I (vw^{2t})|}/{\langle vw^{2t}\rangle_I})$.

We will use the boundedness of the positive operator $P^t_{w,\lambda}$ from  $L^2(u)$ into $L^2(v)$ to bound the expression above. Note that the boundedness of the positive operator is equivalent, by duality, to the following inequality,
$$|\langle v^{\frac12}P^t_{w,\lambda}(u^{-\frac12}f),g\rangle |\leq C_4\,\|f\|_{L^2(\R)}\|g\|_{L^2(\R)}\,.$$

Note that $f\in L^2(\R)$ if and only if $u^{-1/2}f\in L^2(u)$ and similarly $g\in L^2(\R)$ if and only if $v^{1/2}g\in L^2(v^{-1})$ the dual space of $L^2(v)$ under the $L^2$-pairing. The left hand side of the inequality can be written as:
\begin{align*}
\langle v^{\frac12}P^t_{w,\lambda}(u^{-\frac12}f),g\rangle &=\left\langle v^{\frac12}  \sum_{I\in\mathcal{D}}\frac{w^t}{|I|}\lambda_I \langle u^{-\frac12}f\rangle_I\, \mathbbm{1}_I,g\right\rangle\\
&=\left\langle w^{t}v^{\frac12}  \sum_{I\in\mathcal{D}}\frac{\lambda_I}{|I|} \langle u^{-\frac12}f\rangle_I \mathbbm{1}_I,g\right\rangle\\
&=\sum_{I\in\mathcal{D}}\lambda_I\, \langle u^{-\frac12}f\rangle_I\,\langle v^{\frac12}w^tg\rangle_I\,.
\end{align*}
Here the exchange of inner product and summation is legal as we are assuming the operator $P^t_{w,\lambda}$ is bounded.

%
%
We conclude that
\begin{equation}\label{eqn:positive-operator}
\sum_{I\in\mathcal{D}}\lambda_I\, \langle u^{-\frac12}f\rangle_I \, \langle v^{\frac12}w^tg\rangle_I\leq C_4\, \|f\|_{L^2(\R)}\|g\|_{L^2(\R)}\,.
\end{equation}
Therefore, replacing in \eqref{eqn:positive-operator} the function $f$ by $|f|u^{-1/2}$ and the function $g$ by $|g|v^{1/2}$ which are both in $L^2(\R)$, we have the estimate for $\Sigma_4(f,g)$,
\begin{eqnarray*}
|\Sigma_4(f,g)| &\leq &\sum_{I\in\D}\lambda_I \langle u^{-\frac12}(|f|u^{-\frac12})\rangle_I \,\langle v^{\frac12}w^t(v^{\frac12}|g|)\rangle_I \\
& \leq & C_4\,\|u^{-\frac12}f\|_{L^2(\R)}\|v^{\frac12}g\|_{L^2(\R)} 
\, = \, C_4\,\|f\|_{L^2(u^{-1})}\|g\|_{L^2(v)}\,.
\end{eqnarray*}
This completes the proof of Theorem~\ref{thm:2weight-t-HaarMultiplier}.
\end{proof}

\subsection{Reduction to the case $u=v\equiv 1$ and $\sigma\equiv 1$}\label{u=v=1}

When we apply Theorem~\ref{thm:2weight-t-HaarMultiplier} to the case $u=v\equiv 1$
and $\sigma\equiv 1$ we recover or improve upon all known results \cite{KP, Be, BeMP}. Namely,  we recover the  necessary and sufficient conditions on $w$  when $t> 1/2$ ($w\in RH_{2t}^d$) and when $t<0$ ($w\in A_{1-1/(2t)}^d$), and  we improve upon the known sufficient condition  in the  case $0< t \leq 1/2$  ($w^{2t}\in RH^d_1$). The case $t=0$  is the identity operator always bounded on $L^2(\R)$.

\begin{theorem}[The case $u=v\equiv1$]\label{thm:u=v=1}
Given $t\in\R$ and $w$ a weight.
\begin{itemize}
\item[{\rm (a)}] {\rm (\cite[Section 5.2]{KP})} For $t>1/2$ and $t\leq 0$ the $t$-Haar multiplier $T^t_w$ is bounded on $L^2(\R)$  if and only if the following $C_{2t}$ condition holds:
\[[w]_{C_{2t}}:=\sup_{I\in\mathcal{D}}\langle w^{2t} \rangle_I \langle w\rangle^{-2t}_I<\infty.\]
Moreover $\|T^t_w\|^2_{L^2(\R)\to L^2(\R)}\approx [w]_{C_{2t}}$.

\item[{\rm (b)}] For $0< t\leq 1/2$ if there is a constant $C>0$ such that 
\[\frac{1}{|I|}  \sum_{J\in\mathcal{D}: J\subset I} |J| \frac{ |\Delta_J (w^{2t})|^2}{\langle w\rangle_J^{2t}} \leq C_3\, \langle w^{2t} \rangle_I \quad\quad \mbox{for all $I\in\mathcal{D}$,}\]
if and only if $T^t_w$ is bounded on $L^2(\R)$. Moreover $\|T^t_w\|^2_{L^2(\R)\to L^2(\R)}\approx C_3$.
\end{itemize}
\end{theorem}

\begin{proof}
When reducing to the known case $u=v\equiv 1$  in Theorem~\ref{thm:2weight-t-HaarMultiplier}, conditions (ii) and (iv) disappear as $\Delta_I 1=0$, and  conditions (i) and (iii) survive.
In fact condition (i) corresponds to the  $C_{2t}$-class of weights, namely those weights with the property that 
\begin{equation}\label{C2t}
\sup_{I\in\mathcal{D}}\langle w^{2t} \rangle_I \langle w\rangle^{-2t}_I<\infty,
\end{equation}
while condition (iii) becomes
\begin{equation}\label{eqn:Carleson2}
\frac{1}{|I|}  \sum_{J\in\mathcal{D}: J\subset I} |J| \frac{ |\Delta_J (w^{2t})|^2}{\langle w\rangle_J^{2t}} \leq C_3\, \langle w^{2t} \rangle_I.
\end{equation}
Note that condition~\eqref{C2t} can be deduced from the boundedness  of $T_w^t$  on $L^2(\R)$ by testing on Haar functions. More precisely $T^t_wh_I = \langle w\rangle_I^{-t}w^t h_I$, and $\| T^t_wh_I\|^2_{L^2(\R)}= \langle w^{2t}\rangle_I\langle w\rangle_I^{-2t}$, while $\|h_I\|_{L^2(\R)}=1$.
\\

\noindent {\it Proof of} (a). When $t\leq 0$ and $t>1/2$, we would like to ensure that \eqref{C2t} and \eqref{eqn:Carleson2}  are either equivalent conditions or that~\eqref{C2t} implies~\eqref{eqn:Carleson2},  as the $C_{2t}$ condition~\eqref{C2t}   is known to be necessary and sufficient for the boundedness of $T_w^t$ on $L^2(\R )$, see  \cite[Section 5.2]{KP}. 

When $t=0$,  the symbol is just 1 and $T_w^0$ is the identity operator in $L^2(\R )$ which is always bounded. Note that in this case~\eqref{C2t} and~\eqref{eqn:Carleson2} are both trivially true, in one case because $1<\infty$ and in the other case because $\Delta_J1=0$.

When $2t>1$,  condition~\eqref{C2t}   coincides with $RH^d_{2t}$ which, by Proposition~\ref{prop:RHp}, is equivalent to condition~\eqref{eqn:Carleson2}, where we are using $p=2t>1$.

When $t<0$, then condition~\eqref{C2t}   coincides with  $A^d_{1-\frac{1}{2t}}$.  But we  showed  that~\eqref{C2t}  implies~\eqref{eqn:Carleson2} in Proposition~\ref{prop:Ap},
where in this case  $p=1-1/(2t)>1$ and $-1/(p-1)=2t$.  
Therefore~\eqref{eqn:Carleson2} is a superflous hypothesis in the case when $u=v=1$ and $t<0$. 

Thus, in the cases $t\leq 0$ and $2t>1$, condition~\eqref{C2t}  is necessary and sufficient for boundedness  of $T^t_w$ on $L^2(\R)$, as~\eqref{C2t} is necessary and sufficient and the only other non-trivial condition~\eqref{eqn:Carleson2} is either equivalent to~\eqref{C2t}  or a consequence of~\eqref{C2t}. \\


\noindent{\it Proof of} (b). When $0<  2t\leq 1$, $C_{2t}$ condition~\eqref{C2t} is automatically true, hence condition~\eqref{eqn:Carleson2} will  be sufficient.  When $2t=1$, condition~\eqref{C2t} is clearly true as the left-hand-side is equal to one. When $0<2t< 1$, condition~\eqref{C2t}  holds, since $\langle w^{2t}\rangle_I\leq \langle w\rangle_I^{2t}$, by H\"older's inequality with $p=1/(2t)>1$.

Regarding the necessity, we can obtain condition~\eqref{eqn:Carleson2} by testing the adjoint operator $(T^t_w)^*$ on the functions $\mathbbm{1}_Iw^t$, where
$(T^t_w)^*f=\sum_{J\in\mathcal{D}}\langle fw^t,h_J\rangle / \langle w\rangle_J^t h_J.$
More precisely, $(T^t_w)^*(\mathbbm{1}_Iw^t)= \sum_{J\in\mathcal{D}}\langle \mathbbm{1}_Iw^{2t},h_J\rangle / \langle w\rangle_J^t h_J$,  by Plancherel's identity,
\[\|(T^t_w)^*(\mathbbm{1}_Iw^t)\|_{L^2(\R)}^2= \sum_{J\in\mathcal{D}}\frac{\langle |\mathbbm{1}_Iw^{2t},h_J\rangle|^2} {\langle w\rangle_J^{2t}},\]
hence noting that  $\|\mathbbm{1}_Iw^t\|_{L^2(\R)}^2=  \int_{I} w^{2t}(x)\,dx$ we get
\[ \sum_{J\in\mathcal{D}:J\subset I} \frac{ |\langle w^{2t},h_J\rangle|^2}{ \langle w\rangle_J^{2t}} \leq \|T^t_w\|_{L^2(\R)\to L^2(\R)}^2 \int_{I} w^{2t}(x)\,dx.\]
Noting that $|\langle w^{2t},h_J\rangle|^2= |J||\Delta_J(w^{2t})|^2$, condition~\eqref{eqn:Carleson2} now follows upon dividing by $|I|$ with $C_3=\|T^t_w\|_{L^2(\R)\to L^2(\R)}^2$.
\end{proof}

What was known before, in the case $0<2t\leq 1$, was that $w^{2t}\in RH^d_1$ implies boundedness of $T_w^t$ on $L^2(\R )$, see \cite[Section 5.2]{KP}. We warn the reader that in \cite{KP}, $A^d_{\infty}$ was defined to be the union of all $RH_p^d$ for $p>1$, that is what we  call  $RH_1^d$ in this paper.

When $2t=1$, condition~\eqref{eqn:Carleson2} is Buckley's $RH^d_1$ condition for $w$ (Lemma~\ref{lem:BuckleyRH1}).

When $0<2t< 1$,  if $w^{2t}\in RH^d_1$ then~\eqref{eqn:Carleson2} holds because
\[  \frac{1}{|I|}  \sum_{J\in\mathcal{D}: J\subset I} |J| \frac{ |\Delta_J (w^{2t})|^2}{\langle w\rangle_J^{2t}} \leq  \frac{1}{|I|}  \sum_{J\in\mathcal{D}: J\subset I} |J| \frac{ |\Delta_J (w^{2t})|^2}{\langle w^{2t}\rangle_J}  \leq [w^{2t}]_{RH^d_1} \langle w^{2t} \rangle_I.\]
Where the last inequality is by Lemma~\ref{lem:BuckleyRH1} applied to $w^{2t}$ assumed to be in $RH^d_1$. In this case condition~\eqref{eqn:Carleson2} is a sufficient condition for the boundedness of $T^t_w$ that improves upon the priorly known sufficient condition $w^{2t}\in RH_1^d$.
\begin{question} Does there exist a  weight $w$ and some $0<t\leq 1/2$  for which \eqref{eqn:Carleson2} holds but $w^{2t}$ is not in $RH^d_1$? This will confirm that condition~\eqref{eqn:Carleson2} is indeed a weaker condition than $RH^d_1$. 
\end{question}

%
%


%
%
%

\subsection{Reduction to the one weight case $u=v$}\label{u=v}

In the case when $u=v$,  condition (i) in Theorem~\ref{thm:2weight-t-HaarMultiplier} becomes
\begin{equation}\label{u=vCondition(i)}
\sup_{I\in\mathcal{D}} \frac {\langle u^{-1}\rangle_I \langle uw^{2t}\rangle_I}{\langle w\rangle_I^{2t}} <\infty. 
\end{equation}
The reverse of condition (i), namely
\begin{equation}\label{reverseCondition(i)}
\langle w\rangle_I^{2t} \leq  \langle u^{-1}\rangle_I \langle uw^{2t}\rangle_I,
\end{equation}
holds by H\"older's inequality when $t\geq 1$ and when $t\leq 0$,
but this argument does not work  when $0<t<1$. 
With \eqref{reverseCondition(i)} one can verify that (i), (ii) and (iii) imply (iv). So conditions (i)-(iii) are sufficient for the boundedness of $T^t_w$ on $L^2(u)$ when $t\leq 0$ and when $t\geq 1$, as we will show in the proof of Theorem~\ref{thm:1weight-t-HaarMultiplier}.

When $t=0$,  the $t$-Haar multiplier is $T_{\sigma}$, the martingale transform. If in addition $u=v$, then condition~(i) is $u\in A_2^d$, and conditions~(ii) and~(iii) reduce to $u^{-1}\in RH^d_1$ and $u\in RH^d_1$, respectively. To see why, we use that  $\langle u\rangle_I \sim 1/\langle u^{-1}\rangle_I$ for $u\in A_2^d$, and Buckley's characterization of $RH_1^d$ (Lemma~\ref{lem:BuckleyRH1}). If  $u\in A^d_2$ then both $u$ and $u^{-1}$ belong to $RH_1^d$. 
 In other words, in this case, condition~(i) implies conditions~(ii) and~(iii).  Theorem~\ref{thm:1weight-t-HaarMultiplier} recovers the known one weight results for the martingale transform \cite{TV, Wi}, i.e., the martingale transform is uniformly bounded on $L^2(u)$ if and only if $u\in A_2^d$.

\begin{theorem}[One-weight theorem: sufficient conditions]\label{thm:1weight-t-HaarMultiplier}
Given $t\leq 0$ or $t\geq 1 $ and $w$ a weight, then $T_{w,\sigma}^t$ is bounded from $L^2(u)$ into $L^2(u)$ for each choice of signs $\sigma$ if the following hold:
\begin{itemize}
\item[{\rm (i)}]   A two-weight condition:  $\displaystyle{\;\; C_1:=\sup_{I\in\mathcal{D}} \frac {\langle u^{-1}\rangle_I \langle uw^{2t}\rangle_I}{\langle w\rangle_I^{2t}} <\infty.} $   

\item[{\rm (ii)}] A two weight Carleson condition: there is $C_2>0$ such that for all $I\in\mathcal{D}$ 
\[ \frac{1}{|I|}  \sum_{J\in\mathcal{D}: J\subset I} |J| \frac{\langle uw^{2t}\rangle_J |\Delta_J (u^{-1})|^2}{\langle w\rangle_J^{2t}} \leq C_2 \langle u^{-1}\ \rangle_I.\]
In other words, $\left \{\mu_I:=|J| {\langle uw^{2t}\rangle_J |\Delta_J (u^{-1})|^2}/{\langle w\rangle_J^{2t}}\right \}_{I\in\mathcal{D}}$ is a $u^{-1}$-Carleson sequence.

\item[{\rm (iii)}] A dual two weight Carleson condition: there is $C_3>0$ such that for all $I\in\mathcal{D}$ 

\[ \frac{1}{|I|}  \sum_{J\in\mathcal{D}: J\subset I} |J| \frac{\langle u^{-1}\rangle_J |\Delta_J (uw^{2t})|^2}{\langle w\rangle_J^{2t}} \leq C_3 \langle uw^{2t} \rangle_I.\]
In other words, $\left \{\rho_I:= |J| {\langle u^{-1}\rangle_J |\Delta_J (uw^{2t})|^2}/{\langle w\rangle_J^{2t}}\right \}_{I\in\mathcal{D}} $ is a  $uw^{2t}$-Carleson sequence.

\end{itemize}
Moreover $\|T^t_{w,\sigma}f\|_{L^2(u)} \lesssim \sqrt{C_1} + \sqrt{C_2} + \sqrt{C_3} + \sqrt{C_2C_3}$,
\end{theorem}

\begin{proof} By Theorem~\ref{thm:2weight-t-HaarMultiplier}, suffices to show that 
conditions (i)-(iii) imply that the positive operator $P^t_{w,\lambda}$ is bounded on $L^2(u)$, where
\[ P^t_{w,\lambda} f(x) := \sum_{I\in\mathcal{D}} \frac{w^t(x)}{|I|} \lambda_I \langle f\rangle_I \mathbbm{1}_I(x), \]
and 
$\displaystyle{\lambda_I= \frac{|\Delta_I(u^{-1})|}{\langle u^{-1}\rangle_I}\frac{|\Delta_I (uw^{2t})|}{\langle uw^{2t}\rangle_I} \frac{|I|}{\langle w\rangle_I^t}= \mu_I^{1/2}\rho_I^{1/2} {\langle w\rangle_I^t}{\langle u^{-1}\rangle_I^{-3/2} \langle uw^{2t}\rangle_I^{-3/2}}.}$\\

First, note that when $t\geq 1$ and when $t\leq 0$, by H\"older's inequality, it holds that $\langle w\rangle^{t}_I \leq \langle w^t\rangle_I$ (however this is no longer true when $0<t<1$). Therefore in this case the reverse of condition (i) holds by the Cauchy-Schwarz inequality, namely:
\[  1 \leq \frac {\langle u^{-1}\rangle_I \langle uw^{2t}\rangle_I}{\langle w^{t}\rangle_I^{2}}\leq \frac {\langle u^{-1}\rangle_I \langle uw^{2t}\rangle_I}{\langle w\rangle_I^{2t}}.\]
This implies that 
\begin{equation}\label{eqn:lambdaI} 
\lambda_I\leq \mu_I^{1/2}\rho_I^{1/2} \frac{1}{\langle u^{-1}\rangle_I \langle uw^{2t}\rangle_I}.
\end{equation}

Second, by duality, showing the boundedness of $P^t_{w,\lambda}$ on $L^2(u)$  is equivalent to proving that
\[ | \langle P^t_{w,\lambda} (fu^{-1}), gu\rangle | \lesssim \|f\|_{L^2(u^{-1})} \|g\|_{L^2(u)}.\]
We will show this inequality using first \eqref{eqn:lambdaI}, then the Cauchy-Schwarz inequality, and  then the Weighted Carleson's Lemma  (Lemma~\ref{WCL}) applied to the given $u^{-1}$ and $uw^{2t}$-Carleson sequences in conditions (ii) and (iii). We have that,

\begin{eqnarray*} 
 | \langle P^t_{w,\lambda} (fu^{-1}), gu\rangle | & \leq & \sum_{I\in\mathcal{D}}   \lambda_I |\langle fu^{-1}\rangle_I| | \langle w^t gu\rangle_I |\\
 & \leq & \sum_{I\in\mathcal{D}}    \mu_I^{1/2} \frac{\langle |f| u^{-1} \rangle_I}{\langle u^{-1}\rangle_I}  \rho_I^{1/2} \frac{\langle uw^{t} |g|\rangle_I}{\langle uw^{2t}\rangle_I}\\
 & \leq &  \sum_{I\in\mathcal{D}}    \mu_I^{1/2} \langle |f|\rangle_I^{u^{-1}}  \rho_I^{1/2} \langle w^{-t} |g|\rangle_I^{uw^{2t}} \\
 & \leq & \left ( \sum_{I\in\mathcal{D}}    \mu_I  |\langle |f|\rangle_I^{u^{-1}}|^2 \right )^{1/2}  \left ( \sum_{I\in\mathcal{D}}   \rho_I |\langle w^{-t} |g|\rangle_I^{uw^{2t}} |^2 \right )^{1/2}\\
  & \leq & \sqrt{C_2C_3}\, \|M_{u^{-1}}f\|_{L^2(u^{-1})} \| M_{uw^{2t}} (w^{-t}g)\|_{L^2(uw^{2t})} \\
  & \leq & 4\sqrt{C_2C_3}\, \| f\|_{L^2(u^{-1})} \|w^{-t}g\|_{L^2(uw^{2t})}\\
 & = & 4\sqrt{C_2C_3}\,  \| f\|_{L^2(u^{-1})} \|g\|_{L^2(u)}.
\end{eqnarray*}
\end{proof}

There is the case  $0<t<1$  left to analyze. Given what we know when $u=v=1$ it is expected that the arguments for $1/2<t<1$ and for  $0<t\leq 1/2$ may be different. 
We would hope Theorem~\ref{thm:1weight-t-HaarMultiplier} still holds for $1/2<t<1$. Note that Theorem~\ref{thm:two-weight} does hold for all $t\in \R$, in particular for $0<t<1$.
\begin{question} Does Theorem~\ref{thm:1weight-t-HaarMultiplier} hold for some $t\in (0,1)$?
\end{question}

\section{Necessary conditions for two-weight boundedness of $T^t_{w,\sigma}$}\label{sec:necessary}

The conditions known to be necessary for boundedness from $L^2(u)$ into $L^2(v)$ of $T^t_{w,\sigma}$ are testing conditions. The simplest necessary conditions are obtained by testing on Haar functions. First observe that  $T^t_{w,\sigma}(h_I)= \sigma_I \langle w\rangle_I^{-t}w^th_I$.
Therefore the boundedness of $T_{w,\sigma}^t$ from $L^2(u)$ into $L^2(v)$ implies that  for all $I\in\mathcal{D}$
\[ \| \sigma_I \langle w\rangle_I^{-t}w^th_I\|_{L^2(v)}\leq C \|h_I\|_{L^2(u)},\]
where $C=\|T^t_{w,\sigma}\|_{L^2(u)\to L^2(v)}$.
Computing both norms we get
$ \langle w\rangle_I^{-t} \langle w^{2t}v\rangle_I^{1/2}\leq C \langle u\rangle_I^{1/2},$
which after squaring and multiplying appropriately becomes
\begin{eqnarray}\label{eqn:testing}
 \frac{\langle u\rangle_I^{-1} \langle w^{2t}v\rangle_I}{\langle w\rangle_I^{2t}} \leq C^2 .
 \end{eqnarray}
This is close but not exactly  condition (i) in Theorem~\ref{thm:2weight-t-HaarMultiplier}.
But since $1\leq \langle u\rangle_I \langle u^{-1} \rangle_I$, condition (i)  implies the testing condition \eqref{eqn:testing}. But not the other way around unless $u\in A_2^d$. 

We will show in this section that conditions (i)-(iv) are necessary for the uniform (with respect to $\sigma$) boundedness of $T^t_{w,\sigma}$.

When $w\equiv 1$ or $t=0$, the corresponding $t$-Haar multiplier  is the martingale transform $T_{\sigma}$ and  conditions (i)-(iv) are exactly the necessary and sufficient conditions of Nazarov, Treil and Volberg for the uniform boundedness of the martingale transform \cite{NTV1}. This is an indication that the conditions are necessary if they are to hold for all values of $t\in\R$ and for all weights $w$. If we fix the value of $t\in\R$ and a weight $w$, are the conditions still necessary? The answer is yes.

\subsection{The necessity of conditions (i), (ii) and (iii)}\label{sec:necessity(i)-(iii)}

In this section we show that conditions (i), (ii) and (iii) in Theorem~\ref{thm:2weight-t-HaarMultiplier} are necessary for the uniform (on choice of signs $\sigma$) boundedness of $T^t_{w,\sigma}$ from $L^2(u)$ into $L^2(v)$ and for each fixed $t\in\R$ and weight $w$.

\begin{theorem}[Two-weight theorem: necessary conditions (i)-(iii)]\label{thm:necessity(i)-(iii)}
Let $t\in\R$ and $w$ a weight be fixed. If $T^t_{w,\sigma}$ are uniformly (with respect to $\sigma$) bounded from $L^2(u)$ into $L^2(v)$ then conditions {\rm (i)}, {\rm (ii)} and {\rm (iii)} in Theorem~\ref{thm:2weight-t-HaarMultiplier} hold. Moreover
$C_i\lesssim \|T^t_{w,\sigma}\|^2_{L^2(u)\to L^2(v)}$ for $i=1,2,3$.
\end{theorem}
\begin{proof}
The uniform (with respect to $\sigma$) boundedness of the operator 
$$T^t_{w,\sigma}f=\sum_{I\in\mathcal{D}}\sigma_I\left(\frac{w(x)}{\langle w\rangle_I}\right)^t\langle f,h_I\rangle h_I$$
from $L^2(u)$ into $L^2(v)$ is equivalent, by duality, to the existence of a constant $C>0$ such that for all choices of signs $\sigma$, and for all $f,g\in L^2(\R)$ the following estimate holds 
\begin{equation}\label{duality}
|\langle v^{1/2}T_{w,\sigma}^t(u^{-1/2}f),g\rangle | \leq C\|f\|_{L^2(\R)}\|g\|_{L^2(\R)}\, ,
\end{equation}
where $C=\|T^t_{w,\sigma}\|_{L^2(u)\to L^2(v)}$.
 Moreover $\|f\|_{L^2(\R)}=\|u^{-1/2}f\|_{L^2(u)}$ and $\|g\|_{L^2(\R)}=\|v^{1/2}g\|_{L^2(v^{-1})}$. 

The left-hand-side of  inequality~\eqref{duality} can be written as
\begin{align*}
  \langle v^{1/2}T_{w,\sigma}^t(u^{-1/2}f),g\rangle &=\left\langle v^{1/2}\sum_{I\in\mathcal{D}}\sigma_I\left(\frac{w(x)}{\langle w\rangle_I}\right)^t\langle u^{-1/2}f,h_I\rangle h_I,g\right\rangle\\
  &=\left\langle (vw^{2t})^{1/2}\sum_{I\in\mathcal{D}}\frac{\sigma_I}{\langle w\rangle_I^t}\langle u^{-1/2}f, h_I\rangle h_I, g\right\rangle\,.
\end{align*}
Therefore, we can see that the uniform boundedness  from $L^2(u)$ into $L^2(v)$  of the $t$-Haar multipliers $T^t_{w,\sigma}$  is equivalent to the uniform boundedness from $L^2(u)$ to $L^2(vw^{2t})$ of the operators $T_{\sigma}T_{w,t}$ defined by
\begin{equation}\label{TsigmaTwt}
(T_{\sigma}T_{w,t})f:=\sum_{I\in\mathcal{D}}\frac{\sigma_I}{\langle w\rangle_I^t}\langle f, h_I\rangle h_I\,.
\end{equation}
Formally the constant Haar multiplier $T_{\sigma}T_{w,t}$ is the composition of the martingale transform $T_{\sigma}$ and the operator
$T_{w,t}f:= \sum_{I\in\mathcal{D}}{\langle w\rangle_I^{-t}}\langle f, h_I\rangle h_I\,$. That explains the notation we are using in \eqref{TsigmaTwt}. We will treat $T_{\sigma}T_{w,t}$ as a unit, we will not decouple the actions of $T_{\sigma}$ and $T_{w,t}$.

Moreover for each choice of signs $\sigma$, we have 
$$\|T_{w,\sigma}^t\|_{L^2(u)\rightarrow L^2(v)}=\|T_{\sigma} T_{w,t}\|_{L^2(u)\rightarrow L^2(vw^{2t})}\,.$$


Since $\Delta_I f=\langle f\rangle_{I_+}-\langle f\rangle_{I_-}
=\frac{2}{\sqrt{|I|}}\langle f,h_I\rangle$
for a locally integrable function $f$, one can rewrite 
$$T_{\sigma} T_{w,t}f=\frac{1}{2}\sum_{I\in\mathcal{D}}\frac{\sigma_I}{\langle w\rangle_I^t}\sqrt{|I|}\Delta_I f \,h_I\,.$$
The uniform (with respect to $\sigma$) boundedness of  $T_{\sigma}T_{w,t}$ from $L^2(u)$ to $L^2(vw^{2t})$ means there is $C>0$ such that for all $\sigma$ and for all $f\in L^2(u)$
\[\int_{\R}\left| (T_{\sigma} T_{w,t})f(x)\right|^2 v(x)w^{2t}(x)\,dx
 \leq C\int_{\R} \left|f(x)\right|^2 u(x)\,dx\,, \label{testingcondition}\]
 where by previous commentary $C=\|T^t_{w,\sigma}\|^2_{L^2(u)\to L^2(v)}$.
%
%
%
We can compute the expectation $\mathbb{E}_{\sigma}$ with respect to choices of signs $\sigma$ and keep the inequality
\[\mathbb{E}_{\sigma}\left (\| (T_{\sigma}T_{w,t})f\|_{L^2(vw^{2t})}^2 \right )\\
\leq C \int_{\R} \left|f(x)\right|^2 u(x)\,dx,\]
precisely because we have uniform (with respect to the choice of signs  $\sigma$) boundedness.
Khintchine's inequality ensures that
\[
\mathbb{E}_{\sigma}\left (\| (T_{\sigma}T_{w,t})f\|_{L^2(vw^{2t})}^2 \right ) \approx 
 \sum_{I\in\mathcal{D}}\frac{1}{\langle w\rangle_I^{2t}}|I||\Delta_I f|^2 \langle vw^{2t}\rangle_I \,. \]
Thus we have obtained
\begin{equation}
 \sum_{I\in\mathcal{D}}\frac{1}{\langle w\rangle_I^{2t}}|I||\Delta_I f|^2 \langle vw^{2t}\rangle_I  \lesssim C \int_{\R} \left|f(x)\right|^2 u(x)\,dx\,. \label{testingcondition}
\end{equation}

For fixed $J\in\mathcal{D}$, choose $f=u^{-1}\mathbbm{1}_J$ to get
$$ \frac{1}{|J|}\sum_{I\in\mathcal{D}(J)}\frac{1}{\langle w\rangle_I^{2t}}|I||\Delta_I (u^{-1}\mathbbm{1}_J)|^2 \, \langle vw^{2t}\rangle_I
\lesssim C\langle u^{-1}\rangle_J\,.$$ 
Noting that $\Delta_I(u^{-1}\mathbbm{1}_J)=\Delta_I u^{-1}$ for all $I\in\mathcal{D}(J)$, we get  the three weight Carleson condition (ii) in Theorem~\ref{thm:2weight-t-HaarMultiplier}, with  $C\lesssim \|T^t_{w,\sigma}\|^2_{L^2(u)\to L^2(v)}$.

Applying  condition (\ref{testingcondition}), with  $u^{-1}\mathbbm{1}_{J_{\pm}}$ instead of $u^{-1}\mathbbm{1}_J$, we get
\begin{align*}
C\langle u^{-1}\rangle_{J_{\pm}}&\geq   \frac{1}{|J|}\sum_{I\in\mathcal{D}(J)}\frac{1}{\langle w\rangle_I^{2t}}|I||\Delta_I (u^{-1}\mathbbm{1}_{J_{\pm}})|^2 \,\langle vw^{2t}\rangle_I\\
&\geq \frac{1}{|J|}\frac{1}{\langle w\rangle_J^{2t}}|J||\Delta_J(u^{-1}\mathbbm{1}_{J_{\pm}})|^2\, \langle vw^{2t}\rangle_J\\
&=\frac{1}{\langle w\rangle_J^{2t}}\langle u^{-1}\rangle_{J_{\pm}}^2\,\langle vw^{2t}\rangle_J\,,
\end{align*}
where the last equality  is due to $|\Delta_J(u^{-1}\mathbbm{1}_{J_{\pm}})|=\langle u^{-1}\rangle_{J_{\pm}}$. Then

$$\langle u^{-1}\rangle_{J_{\pm}}\, \langle vw^{2t} \rangle_J  \leq C \langle w\rangle_J^{2t}$$
Adding the above inequality for $J_+$ and $J_-$, we get  condition (i) in Theorem~\ref{thm:2weight-t-HaarMultiplier},
$$\frac{\langle u^{-1}\rangle_J\, \langle vw^{2t}\rangle_J}{\langle w\rangle_J^{2t}}\leq C\,,$$
where $C\lesssim\|T^t_{w,\sigma}\|^2_{L^2(u)\to L^2(v)}$.

To get  condition (iii) in Theorem~\ref{thm:2weight-t-HaarMultiplier} we proceed similarly looking at the adjoint operator $(T^t_{w,\sigma})^*: L^2(v^{-1})\to L^2(u^{-1})$ given by
\[ (T^t_{w,\sigma})^*g= \sum_{I\in\mathcal{D}}\sigma_I\left(\frac{\langle w^tf,h_I\rangle }  {\langle w\rangle_I}\right)^th_I.\]
One can verify that 
$$\|(T_{w,\sigma}^t)^*\|_{L^2(v^{-1})\rightarrow L^2(u^{-1})}=\|T_{\sigma} T_{w,t}\|_{L^2(v^{-1}w^{-2t})\rightarrow L^2(u^{-1})}\,.$$
Exactly the same calculation we did to obtain \eqref{testingcondition}, simply interchanging the role of $u$ with $v^{-1}w^{-2t}$ and the role $vw^{2t}$ with $u^{-1}$, will give
\begin{equation}
 \sum_{I\in\mathcal{D}}\frac{1}{\langle w\rangle_I^{2t}}|I||\Delta_I g|^2 \langle u^{-1}\rangle_I  \leq C \int_{\R} \left|g(x)\right|^2 v^{-1}(x)w^{-2t}(x)\,dx\,. \label{dualtestingcondition}
\end{equation}
For a given $J\in\mathcal{D}$ and replacing $g=vw^{2t}\mathbbm{1}_J$ in \eqref{dualtestingcondition} will give condition (iii), where $C\lesssim\|T^t_{w,\sigma}\|^2_{L^2(u)\to L^2(v)}$. The details are left to the reader.
\end{proof}

As a corollary of Theorem~\ref{thm:necessity(i)-(iii)} and Theorem~\ref{thm:1weight-t-HaarMultiplier}, we conclude that in the case $u=v$ and $t\leq 0$ or $t\geq 1$ conditions (i)-(iii) are necessary and sufficient for the uniform (with respect to the choice of signs $\sigma$) boundedness of $T^t_{w,\sigma}$ on $L^2(u)$, which is precisely Theorem~\ref{thm:one-weight} (one-weight theorem) stated in the introduction.

\subsection{The necessity of the condition (iv)}\label{sec:necessity(iv)}

In this section we show that the uniform boundedness (on the choice of $\sigma$) of the $t$-Haar multipliers from $L^2(u)$ into $L^2(v)$ implies condition (iv)  in Theorem~\ref{thm:2weight-t-HaarMultiplier} , namely the boundedness of the positive operator $P^t_{w,\lambda}$ from $L^2(u)$ into $L^2(v)$, where
$$P^t_{w,\lambda}f(x) := \sum_{I\in\mathcal{D}} \frac{w^t(x)}{|I|}\lambda_I \langle f\rangle_I \mathbbm{1}_I(x), \quad  \mbox{and}  \quad \lambda_I:= 
\frac{|\Delta_I(vw^{2t})|}{\langle vw^{2t}\rangle_I}\frac{|\Delta_Iu^{-1}|}{\langle u^{-1}\rangle_I} \frac{|I|}{\langle w\rangle_I^t}. $$
%

First  we recall the Bilinear Embedding Theorem of Nazarov, Treil and Volberg \cite{NTV1}. This theorem says that under some Sawyer type testing conditions a bilinear embedding theorem will hold. We will then show that the uniform (on the choice of signs $\sigma$) boundedness of $T^t_{w,\sigma}$ implies those conditions and hence the boundedness of $P^t_{w,\lambda}$.  The proof of this bilinear embedding theorem uses an intricate Bellman function argument. 

We are stating  the bilinear embedding theorem for the pair of weight of interest to us: $u^{-1}$ and $vw^{2t}$.

\begin{theorem}[Bilinear Embedding Theorem \cite{NTV1} Section 2] \label{thm:bilinear-embedding}
Let $\{\alpha_I\}_{I\in\mathcal{D}}$ be a sequence of non-negative numbers, and let 
$$T_0f:=\sum_{I\in\mathcal{D}} \frac{\alpha_I}{|I|}\langle f\rangle_I \mathbbm{1}_I(x), $$ 
Let $u^{-1}$ and $vw^{2t}$ be two weights. Then  the following are equivalent:
\begin{itemize}
 \item[{\rm (a)}] The operator $v^{1/2}w^t \, T_0\,u^{-1/2}$ is bounded on $L^2(\R)$.
\item[{\rm (b)}] There is $C>0$ such that for all non negative $f,g\in L^2(\R )$ the following bilinear inequality holds 
\[ \sum_{I\in\mathcal{D}} \frac{\alpha_I}{|I|^2}\langle fu^{-1/2}\rangle_I \langle gv^{1/2}w^t \rangle_I \leq C \|f\|_{L^2(\R)}\|g\|_{L^2(\R)}\].
 \item[{\rm (c)}]There is a constant $C'>0$ such that for all $I\in \mathcal{D}$
 \begin{itemize}
 \item[{\rm 1.}]  $\displaystyle{ \frac{1}{|I|}\int_I |T_0(u^{-1}\mathbbm{1}_I)(x)|^2v(x)w^{2t}(x) \,dx \leq C' \langle u^{-1}\rangle_I}$, and
 \item[{\rm 2.}] $\displaystyle{ \frac{1}{|I|}\int_I |T_0^*(vw^{2t}\mathbbm{1}_I)(x)|^2 u^{-1}(x) \,dx \leq C' \langle vw^{2t}\rangle_I}$.
  \end{itemize}
 \end{itemize}

\end{theorem}
We will use this theorem for $\displaystyle{\alpha_I=\lambda_I= \frac{|\Delta_I(u^{-1})|}{\langle u^{-1}\rangle_I}\frac{|\Delta_I (vw^{2t})|}{\langle vw^{2t}\rangle_I} \frac{|I|}{\langle w\rangle_I^t}}$.  \\


We are now ready to show  the necessity of condition (iv) in Theorem~\ref{thm:2weight-t-HaarMultiplier} given the uniform two-weight boundedness of the $t$-Haar multipliers. With this, Theorem~\ref{thm:two-weight} advertised in the Introduction (Section~\ref{sec:introduction}) will be proved.

\begin{theorem}[Two-weight theorem: necessary condition (iv)]\label{thm;necessity(iv)}
Given  $t\in\R$ and a   weight $w$.  If  the signed $t$-Haar multipliers $T^t_{w,\sigma}$ are uniformly (with respect to $\sigma$) bounded from $L^2(u)$ into $L^2(v)$ then condition {\rm (iv)} in Theorem~\ref{thm:2weight-t-HaarMultiplier} holds. Moreover $\|P^t_{w,\lambda}\|_{L^2(u)\to L^2(v)} \lesssim \|T^t_{w,\sigma}\|_{L^2(u)\to L^2(v)}$.
\end{theorem}
\begin{proof}
The work in the previous sections showed that the first three conditions in Theorem~\ref{thm:2weight-t-HaarMultiplier}  are necessary for the uniform boundedness of $T^t_{w,\sigma}$ (Theorem~\ref{thm:necessity(i)-(iii)}), and together with condition (iv), they are sufficient conditions (Theorem~\ref{thm:2weight-t-HaarMultiplier}). Moreover we showed in \eqref{t1} that for all choices of signs $\sigma$, 
\begin{align}
 |\langle T^t_{w,\sigma}(fu^{-1/2}), gv^{1/2}\rangle|&= \left | \sum_{I\in\D}\frac{\sigma_I}{\langle w\rangle_I^t} \langle u^{-1/2} f, h_I\rangle  \langle g v^{1/2}w^t, h_I\rangle  \right | \nonumber\\
&\leq \sum_{I\in\D}\frac{1}{\langle w\rangle_I^t}|\langle u^{-1/2} f, h_I\rangle | |\langle g v^{1/2}w^t, h_I\rangle| =: \Sigma (f,g). \label{eqn:sigmafg}
\end{align}
Here we are replacing $fu^{-1}\in L^2(u)$ and $gv\in L^2(v^{-1})$ in \eqref{t1} by $fu^{-1/2}\in L^2(u)$ and by $gv^{1/2}\in L^2(v^{-1})$ or equivalently $f,g\in L^2(\R)$.

In particular, given $f,g\in L^2(\R )$, we can select  each $\sigma_I$ so that the inequality \eqref{eqn:sigmafg} becomes an equality, in other words, for all $f, g\in L^2(\R)$
\[ \sup_{\sigma} |\langle T^t_{w,\sigma}(fu^{-1/2}), gv^{1/2}\rangle| = \Sigma(f,g).\]
Thus, the uniform boundedness of $T^t_{w,\sigma}$ from $L^2(u)$ into $L^2(v)$ is equivalent to the following bilinear estimate
\begin{equation}\label{bilinear-estimate}
\Sigma(f,g)=  \sum_{I\in\D}\frac{1}{\langle w\rangle_I^t}|\langle u^{-1/2} f, h_I\rangle | |\langle g v^{1/2}w^t, h_I\rangle| \leq C \|f\|_{L^2(\R )}\|g\|_{L^2(\R )}.
\end{equation}

In the proof of Theorem~\ref{thm:2weight-t-HaarMultiplier} we showed that 
\[ \langle T^t_{w,\sigma}(fu^{-1/2}), gv^{1/2}\rangle = \widetilde{\Sigma}_1(f,g)+\widetilde{\Sigma}_2(f,g)+\widetilde{\Sigma}_3(f,g)+\widetilde{\Sigma}_4(f,g).\]
where $\widetilde{\Sigma}_i(f,g)=\Sigma_i(fu^{1/2},gv^{-1/2})$ for $i=1,2,3,4$, and $\Sigma_i(f,g)$ where defined in~\eqref{eqn:decompositon4terma}. By the triangle inequality 
\[| \langle T^t_{w,\sigma}(fu^{-1/2}), gv^{1/2}\rangle -\widetilde{\Sigma}_4(f,g)  |\leq | \widetilde{\Sigma}_1(f,g)| + | \widetilde{\Sigma}_2(f,g)|+| \widetilde{\Sigma}_3(f,g)| .\] 
We also showed, in the proof of Theorem~\ref{thm:2weight-t-HaarMultiplier}, that if conditions (i), (ii) and (iii) hold then
\[  | \widetilde{\Sigma}_1(f,g)| + | \widetilde{\Sigma}_2(f,g)|+| \widetilde{\Sigma}_3(f,g)|\leq C \|f\|_{L^2(\R)}\|g\|_{L^2(\R)},\]
where $C\lesssim \|T^t_{w,\sigma}\|_{L^2(u)\to L^2(v)}$ and with the caveat that $f, g\in L^2(\R)$ instead of being in $L^2(u^{-1})$ and $L^2(v)$.
Notice that we just showed in Theorem~\ref{thm:necessity(i)-(iii)} that the uniform boundedness of $T^t_{w,\sigma}$ implies conditions (i), (ii), and (iii). Therefore the uniform boundedness of $T^t_{w,\sigma}$ from $L^2(u)$ into $L^2(v)$ is equivalent to 
the  estimate 
\[ |\widetilde{\Sigma}_4(f,g)|\leq C \|f\|_{L^2(\R)}\|g\|_{L^2(\R)},\]
where $C\lesssim \|T^t_{w,\sigma}\|_{L^2(u)\to L^2(v)}$.
Note that $\widetilde{\Sigma}_4(f,g)$ does depend on $\sigma$ whereas the estimate does not. We will select $\{\sigma_I: I\in\mathcal{D}\}$ so that all the terms in the sum are positive and to emphasize the dependence on this choice of $\sigma$  in what follows, we will write $\widetilde{\Sigma}_4^{\sigma}(f,g)$. For that particular choice of  signs $\sigma_I$,
\begin{eqnarray*}
|\widetilde{\Sigma}^{\sigma}_4(f,g)| &= & \left |\sum_{I\in\D}\frac{\sigma_I}{\langle w\rangle_I^t}\left\langle u^{-1/2}f, \beta_I^{u^{-1}}\frac{\mathbbm{1}_I}{\sqrt{|I|}}\right\rangle \left\langle \beta_I^{vw^{2t}}\frac{\mathbbm{1}_I}{\sqrt{|I|}}, gv^{1/2}w^t\right\rangle \right| \\
&=&\sum_{I\in\D}\frac{1}{\langle w\rangle_I^t}\left|\left\langle u^{-1/2}f, \beta_I^{u^{-1}}\frac{\mathbbm{1}_I}{\sqrt{|I|}}\right\rangle\right|\left|\left\langle \beta_I^{vw^{2t}}\frac{\mathbbm{1}_I}{\sqrt{|I|}}, gv^{1/2}w^t\right\rangle\right|.
\end{eqnarray*}
But now instead of the bounds on the $\beta$'s  used in the proof of Theorem~\ref{thm:2weight-t-HaarMultiplier}, we would like  to use the exact formulas \eqref{eqn:alpha-beta} to get the following equality
\begin{eqnarray}
 |\widetilde{\Sigma}^{\sigma}_4(f,g)| & = & \frac{1}{4}\sum_{I\in\D}\frac{|I|}{\langle w\rangle_I^t}\frac{|\Delta_I(u^{-1}|)}{\langle u^{-1}\rangle_I}\frac{|\Delta_I(vw^{2t})|}{\langle vw^{2t}\rangle_I} \left |\langle u^{-1/2}f\rangle_I\right |\left | \langle gv^{1/2}w^t\rangle_I\right | \nonumber\\
 & = & \frac{1}{4}\sum_{I\in\D} \lambda_I  \left |\langle u^{-1/2}f\rangle_I\right |\left | \langle gv^{1/2}w^t\rangle_I\right | \label{eqn:(IV)}\, .
\end{eqnarray}
If $f$  and $g$ are non-negative functions in $L^2(\R)$, then we can remove the absolute values in \eqref{eqn:(IV)} to  conclude that if $T^t_{w,\sigma}$ are uniformly bounded from $L^2(u)$ into $L^2(v)$ then for all positive $f,g\in L^2(\R)$ 
\[ \widetilde{\Sigma}(f,g):=\frac14\sum_{I\in\D} \lambda_I  \langle u^{-1/2}f\rangle_I \langle gv^{1/2}w^t\rangle_I \leq C\|f\|_{L^2(\R)}\|g\|_{L^2(\R )},\]
where $C\lesssim \|T^t_{w,\sigma}\|_{L^2(u)\to L^2(v)}$.
By the bilinear imbedding theorem of Nazarov, Treil and Volberg (Theorem~\ref{thm:bilinear-embedding}),  we conclude that  the positive operator $v^{1/2}w^t\,T_0\,u^{-1/2}$  is bounded on $L^2(\R)$, or equivalently the positive operator $v^{1/2}\,P^t_{w,\lambda}\, u^{-1/2}$ is bounded  on $L^2(\R)$, or equivalently the positive operator $P^t_{w,\lambda}$ is bounded from $L^2(u)$ into $L^2(v)$. We have shown that the uniform (on the signs $\sigma$) boundedness of the $t$-Haar multipliers $T^t_{w,\sigma}$ from $L^2(u)$ into $L^2(v)$ implies condition~(iv) in Theorem~\ref{thm:2weight-t-HaarMultiplier}, furthermore 
$\|P^t_{w,\lambda}\|_{L^2(u)\to L^2(v)} \lesssim \|T^t_{w,\sigma}\|_{L^2(u)\to L^2(v)} .$
\end{proof}

\section{Necessary and sufficient testing conditions for the boundedness of each individual $T^t_{w,\sigma}$}\label{individual}
First, we recall some definition and theorems from \cite{NTV2}. 
\begin{definition}
A band operator on $\mathbb{R}$ is a bounded operator on $L^2(\mathbb{R})$ whose matrix in the Haar basis has the band structure, meaning there exist $r\in\mathbb{Z}_+$ such that $\langle Th_I,h_J\rangle=0$ for all Haar function $h_I, h_J$ such that $d_{tree}(I,R)>r$.  
\end{definition}


Is the $t$-Haar multiplier $\displaystyle T_{w,\sigma}^tf=\sum_{I\in\mathcal{D}} \sigma_I\left(\frac{w(x)}{\langle w\rangle_I}\right)^t\langle f, h_I\rangle h_I$ a band operator? Since $\langle T^t_{w,\sigma}h_I,h_J\rangle=\sigma_I\langle w\rangle^{-t}_I\langle w^th_I,h_J\rangle\neq 0$, in general it is not a band operator on $\mathbb{R}$. However,  $\langle T_{\sigma}T_{w,t}h_I,h_J\rangle=\sigma_I \langle w\rangle_I^{-t}\langle h_I,h_J\rangle$ implies that the  operator $T_{\sigma}T_{w,t}$, a constant Haar multiplier defined in~\eqref{TsigmaTwt}, is a band operator with $r=0\,.$   

In this section we denote $\mathcal{M}_u$ the operator of multiplication by a function $u$, that is $\mathcal{M}_uf=uf$ and $\mathcal{M}^t_u:=M_{u^t}$. The boundedness of $T$ from $L^2(u)$ into $L^2(v)$ is equivalent to the boundedness of $\mathcal{M}_{v^{1/2}}T \mathcal{M}_{u^{-1/2}}=\mathcal{M}^{1/2}_vT\mathcal{M}^{1/2}_{u^{-1}}$ from $L^2(\R)$ into itself.
We will state the next theorems so that the domain of $T$ is $L^2(u)$, note that  in \cite{NTV2} the domain of $T$ is $L^2(u^{-1})$.

\begin{theorem}[\hskip -.005in\cite{NTV2}, Theorem 1.2]\label{thm:ntv2007-1}
Given weights $u,v$ so that $u^{-1}$ is also a weight, and $T$ a band operator. Then the operator $\mathcal{M}_v^{1/2}T\mathcal{M}_{u^{-1}}^{1/2}$ extends to a bounded operator in $L^2(\R )$ if and only if there is a constant $C>0$ such that for all dyadic intervals $I$ the following two inequalities hold:
\begin{align}
\int_{\mathbb{R}}|T(\mathbbm{1}_I u^{-1})(x)|^2 v(x)\,dx&\leq C\int_I u^{-1}(x)\, dx \label{ineq:NTV2_test_u} \\ 
\int_{\mathbb{R}}|T^{\ast}(\mathbbm{1}_I v)(x)|^2 u^{-1}(x)\,dx&\leq C\int_I v(x)\, dx \label{ineq:NTV2_test_v}
\end{align}
Moreover, the norm of the operator $\mathcal{M}_v^{1/2}T\mathcal{M}_{u^{-1}}^{1/2}$ can be estimated by a constant depending only on the number $r$ in the definition of the band operator, and the constant $C$ in \eqref{ineq:NTV2_test_u} and \eqref{ineq:NTV2_test_v}. 
\end{theorem}
The conditions of Theorem \ref{thm:ntv2007-1} can be relaxed a bit by integrating only over the intervals $I$.
\begin{theorem}[\hskip -.005in\cite{NTV2}, Theorem 1.5]\label{thm:ntv2007-2}
Given $r\geq 0$, weights $u,v$ so that $u^{-1}$ is also a weight, and $T$  a $r$-band operator. Then the operator $\mathcal{M}_v^{1/2}T\mathcal{M}_{u^{-1}}^{1/2}$ extends to a bounded operator in $L^2(\R )$ if and only if  the following two conditions hold:
\begin{itemize}
\item[{\rm (a)}] For all dyadic intervals $I,J$ satisfying $2^{-r}|I|\leq |J|\leq 2^r|I|$, we have that,
$$|\langle T\mathbbm{1}_I,\mathbbm{1}_J\rangle |\leq C \sqrt{u^{-1}(I)v(I)}.$$
\item[{\rm (b)}] For all dyadic intervals $I$, we have that,
$$\int_{I}|T(\mathbbm{1}_I u^{-1})(x)|^2 v(x)\, dx \leq Cu^{-1}(I),\quad \int_{I}|T^{\ast}(\mathbbm{1}_Iv)(x)|^2 u^{-1}(x) \, dx\leq Cv(I).$$
\end{itemize}
\end{theorem}
By using the above theorems, we can get the following proposition.

\begin{proposition}
	For the operators $T^t_{w,\sigma}$ and $T_{\sigma}T_{w,t}$,   the following are equivalent:

\begin{itemize}
\item[{\rm (1)}] $T^t_{w,\sigma}$ is a bounded operator from $L^2(u)$ into $L^2(v)$.
\item[{\rm (2)}] $T_{\sigma}T_{w,t}$ is a bounded operator from $L^2(u)$ into $L^2(vw^{2t})$.  
\item[{\rm (3)}] There is a constant $C>0$ such that $\langle v^{1/2}T^t_{w,\sigma}(u^{-1/2}f),g\rangle \leq C\|f\|_{L^2(\R)}\|g\|_{L^2(\R)}$.
\item[{\rm (4)}] There is a constant $C>0$ such that $\langle v^{1/2}w (T_{\sigma}T_{w,t})(u^{-1/2}f),g\rangle \leq C\|f\|_{L^2(\R)}\|g\|_{L^2(\R)}$.
\item[{\rm (5)}] For all dyadic intervals $I$, the following two inequalities hold:\\
$$\int_{\mathbb{R}}\left|\sum_{K\in\mathcal{D}}\frac{\sigma_I}{\langle w\rangle_K^t}\langle \mathbbm{1}_I u^{-1}, h_K\rangle h_K(x)\right|^2 v(x)\, w^{2t}(x)\, dx\leq Cu^{-1}(I)\,,$$
$$\int_{\mathbb{R}}\left|\sum_{K\in\mathcal{D}}\frac{\sigma_I}{\langle w\rangle_K^t}\langle \mathbbm{1}_I vw^{2t}, h_K\rangle h_K(x)\right|^2 u^{-1}(x)\, dx\leq C(vw^{2t})(I)\,.$$
\item[{\rm (6)}] The following two conditions hold:
\begin{itemize}
\item[{\rm (a)}] For all dyadic intervals $I$ and $J$ satisfying $|J|=|I|$, we have that,
$$\left|\left\langle \sum_{K\in\mathcal{D}}\frac{\sigma_I}{\langle w\rangle_K^t}\langle \mathbbm{1}_I,h_K\rangle h_K,\mathbbm{1}_J\right\rangle\right|\leq C \sqrt{u^{-1}(I)\, vw^{2t}(J)}\,.$$  
\item[{\rm (b)}] For all dyadic intervals $I$, we have that,
$$\int_{I}\left|\sum_{K\in\mathcal{D}}\frac{\sigma_I}{\langle w\rangle_K^t}\langle \mathbbm{1}_I u^{-1}, h_K(x)\rangle h_K\right|^2 v(x)\,w^{2t}(x)\, dx\leq Cu^{-1}(I)\,,$$
$$\int_{I}\left|\sum_{K\in\mathcal{D}}\frac{\sigma_I}{\langle w\rangle_K^t}\langle \mathbbm{1}_I vw^{2t}, h_K(x)\rangle h_K\right|^2 u^{-1}(x)\, dx\leq C(vw^{2t})(I)\,.$$
 \end{itemize}
\end{itemize}

\end{proposition}
\begin{proof}
(1) $\Leftrightarrow$ (2) is due to the equality 
\begin{equation}\|T_{w,\sigma}^t\|_{L^2(u)\rightarrow L^2(v)}=\|T_{\sigma}T_{w,t}\|_{L^2(u)\rightarrow L^2(vw^{2t})}\,.\label{equalnorm} \end{equation} (1) $\Leftrightarrow$ (3) and (2) $\Leftrightarrow$ (4) hold by duality.    Theorem \ref{thm:ntv2007-1} implies  (2) $\Leftrightarrow$ (5) and Theorem \ref{thm:ntv2007-2}  implies that (2) $\Leftrightarrow$ (6), since $T_{\sigma}T_{w,t}$ is a band operator with $r=0$.
\end{proof}

\end{document}